\RequirePackage{fix-cm}
\documentclass[smallextended,envcountsame]{svjour3}       
\smartqed  
\usepackage{graphicx}
\usepackage{mathptmx}      
\usepackage{amsmath}
\usepackage{amssymb}
\usepackage{array}
\usepackage{graphicx}

\newcommand{\dd}{{\rm d}}
\numberwithin{equation}{section} 
\numberwithin{theorem}{section} 

\begin{document}

\title{The Riemann problem for the multidimensional isentropic system of gas dynamics is ill-posed if it contains a shock
}

\titlerunning{The Riemann problem for the multidimensional isentropic system of gas dynamics is ill-posed}        

\author{Simon Markfelder \and 
	Christian Klingenberg 
}


\institute{S. Markfelder \and C. Klingenberg \at 
              Dept. of Mathematics, W{\"u}rzburg University, Emil-Fischer-Str. 40, 97074 W{\"u}rzburg, Germany \\
              \email{simon.markfelder@mathematik.uni-wuerzburg.de} \\ 
              \email{klingen@mathematik.uni-wuerzburg.de}          
}

\date{Received: date / Accepted: date}

\maketitle

\begin{abstract}
In this paper we consider the isentropic compressible Euler equations in two space dimensions together with particular initial data. This data consists of two constant states, where one state lies in the lower and the other state in the upper half plane. The aim is to investigate if there exists a unique entropy solution or if the convex integration method produces infinitely many entropy solutions. For some initial states this question has been answered by E. Feireisl and O. Kreml \cite{feir15}, and also G.-Q. Chen and J. Chen \cite{chen}, where there exists a unique entropy solution. For other initial states E. Chiodaroli, O. Kreml and C. De Lellis, \cite{chio14} and \cite{chio15}, showed that there are infinitely many entropy solutions. For still other initial states the question on uniqueness remained open and this will be the content of this paper. This paper can be seen as a completion of the aforementioned papers by showing that the solution is non-unique in all cases (except if the solution is smooth).
\end{abstract}

\section{Introduction} 

\subsection{Basic notions}

We consider the $2$-dimensional isentropic compressible Euler equations 
\begin{equation}
\begin{split}
\partial_t \varrho + \text{div}_x (\varrho\,v) &= 0, \\
\partial_t (\varrho\,v) + \text{div}_x(\varrho\,v\otimes v)+\nabla_x\left[p(\varrho)\right] &= 0,
\end{split}
\label{eq:euler}
\end{equation}
where the density $\varrho=\varrho(t,x)\in\mathbb{R}^+$ and the velocity $v=v(t,x)\in\mathbb{R}^2$ are functions of the time $t\in[0,\infty)$ and the position $x=(x_1,x_2)\in\mathbb{R}^2$. 

Additionally we consider the polytropic pressure law $p(\varrho)=K\,\varrho^\gamma$ with a constant $K\in\mathbb{R}^+$ and the adiabatic coefficient $\gamma\geq1$. In particular $p''(\varrho)\geq 0$ for all $\varrho>0$, i.e. $p$ is a convex function. 

We are interested in solutions to the Cauchy problem consisting of the Euler system \eqref{eq:euler} and the initial data
\begin{equation}
\begin{split}
\varrho(0,x) &= \varrho_0(x), \\
v(0,x) &= v_0(x).
\end{split}
\label{eq:initial}
\end{equation}

First we will clarify what we understand by the notion ``solution''.
\begin{definition} \emph{(weak solution)}
	A weak solution to the Cauchy problem \eqref{eq:euler}, \eqref{eq:initial} is a pair of functions $(\varrho,v)\in L^\infty([0,\infty)\times\mathbb{R}^2,\mathbb{R}^+\times\mathbb{R}^2)$ such that for all test functions $(\psi,\phi)\in C_c^\infty([0,\infty)\times\mathbb{R}^2,\mathbb{R}\times\mathbb{R}^2)$ the following identities hold:
	\begin{align*}
	\int_0^\infty\int_{\mathbb{R}^2}\big(\varrho\,\partial_t\psi + \varrho\,v\cdot\nabla_x\psi\big)\dd x\,\dd t + \int_{\mathbb{R}^2} \varrho_0(x)\,\psi(0,x)\,\dd x \ &=\ 0, \\
	\int_0^\infty\int_{\mathbb{R}^2}\big(\varrho\,v\cdot\partial_t\phi + \varrho\,v\otimes v:D_x\phi + p(\varrho)\,\text{div}_x\phi\big)\dd x\,\dd t\qquad\quad& \\
	+ \int_{\mathbb{R}^2} \varrho_0(x)\,v_0(x)\cdot\phi(0,x)\,\dd x\ &=\ 0. 
	\end{align*}
	\label{defn:weak}
\end{definition}

Let $\varepsilon$ denote the internal energy which is given by $p(\varrho)=\varrho^2\,\varepsilon'(\varrho)$. In the case of polytropic pressure law one gets $\varepsilon(\varrho)=\frac{K\,\varrho^{\gamma-1}}{\gamma-1}$ if $\gamma>1$ and $\varepsilon(\varrho)=K\,\log(\varrho)$ if $\gamma=1$.

\begin{definition} \emph{(admissible weak solution or entropy solution)}
	A weak solution is admissible if for every non-negative test function $\varphi\in C_c^\infty([0,\infty)\times\mathbb{R}^2,\mathbb{R}^+_0)$ the following inequality is fulfilled:
	\begin{align*}
	\int_0^\infty\int_{\mathbb{R}^2}\Bigg(\bigg(\varrho\,\varepsilon(\varrho)+\varrho\frac{|v|^2}{2}\bigg)\,\partial_t\varphi + \bigg(\varrho\,\varepsilon(\varrho)+\varrho\,\frac{|v|^2}{2}+p(\varrho)\bigg)v\cdot\nabla_x \varphi\Bigg)\dd x\,&\dd t \\
	+ \int_{\mathbb{R}^2} \bigg(\varrho_0(x)\,\varepsilon(\varrho_0(x))+\varrho_0(x)\,\frac{|v_0(x)|^2}{2}\bigg)\varphi(0,x)\,\dd x\ &\geq\ 0.
	\end{align*}
	\label{defn:adm}
\end{definition}

\subsection{Initial data considered in this paper} 

We consider initial data of the following type, namely
\begin{equation}
\begin{split}
\varrho(0,x) &= \varrho_0(x):=\left\{
\begin{array}[c]{ll}
\varrho_- & \text{ if }x_2<0 \\
\varrho_+ & \text{ if }x_2>0
\end{array}
\right. ,\\
v(0,x) &= v_0(x):=\left\{
\begin{array}[c]{ll}
v_- & \text{ if }x_2<0 \\
v_+ & \text{ if }x_2>0
\end{array}
\right. ,
\end{split}
\label{eq:ourinit}
\end{equation}
where $\varrho_\pm\in\mathbb{R}^+$ and $v_\pm\in\mathbb{R}^2$ are constants. We denote the components of the velocities as $v_-=(v_{-\,1},v_{-\,2})^T$, resp. $v_+=(v_{+\,1},v_{+\,2})^T$. Furthermore we suppose that $v_{-\,1}=v_{+\,1}$, which means that the component of the velocity which is parallel to the discontinuity is equal on both sides of the discontinuity. In other words the problem under consideration is a one-dimensional Riemann problem extended to two dimensions.
\begin{figure}[hbt]
	\centering
	\includegraphics[width=0.6\textwidth]{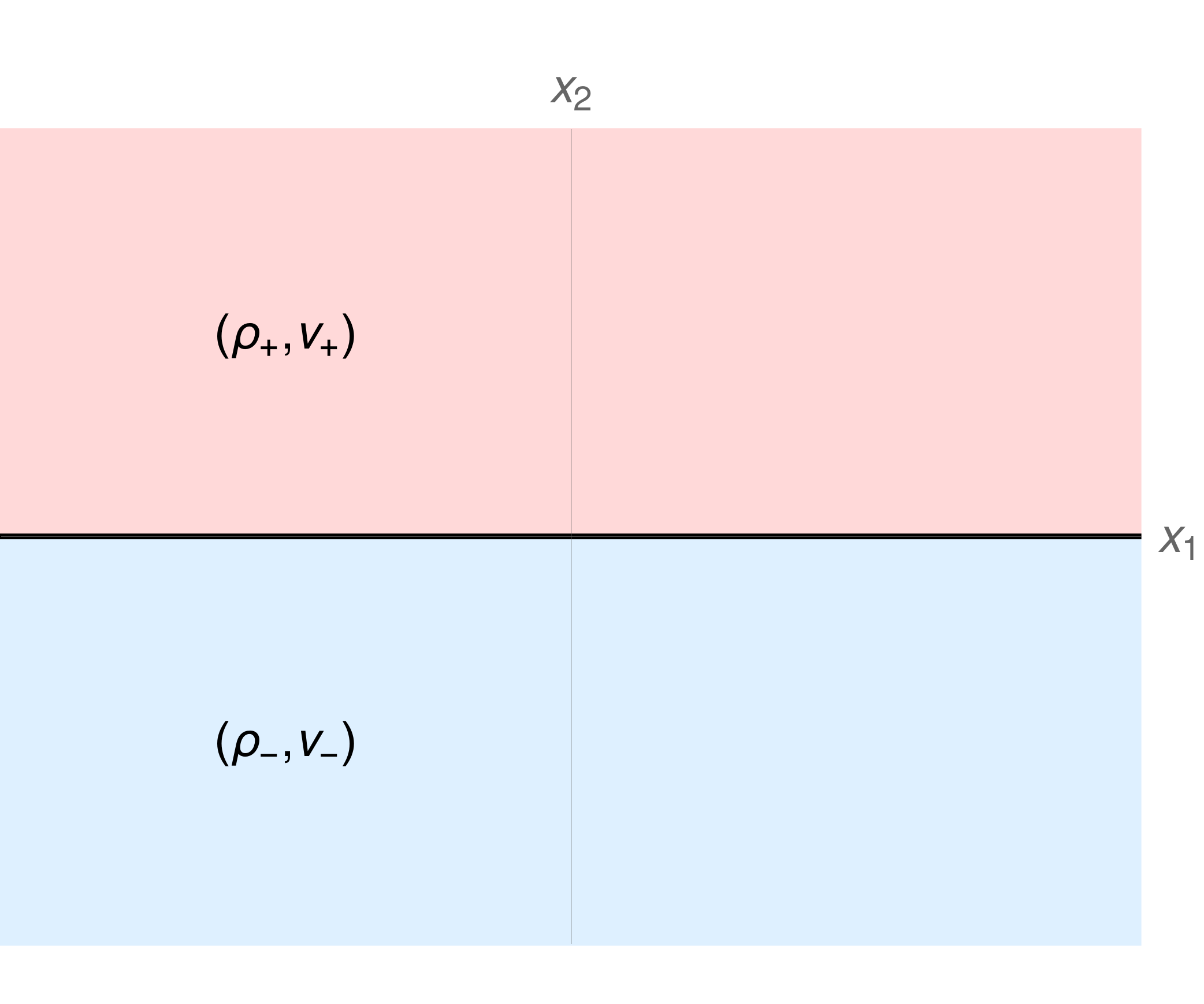}
	\caption{Initial data considered in this paper}
	\label{fig:ourinit}
\end{figure}
The initial data is illustrated in figure \ref{fig:ourinit}. 

Solving the one-dimensional Riemann problem that corresponds to problem \eqref{eq:euler}, \eqref{eq:ourinit} and extending the solution to two space dimensions yields an admissible weak solution to the two-dimensional problem \eqref{eq:euler}, \eqref{eq:ourinit}. We will denote this solution as \emph{standard solution}. 

\begin{proposition} \emph{(see \cite[Lemma 2.4]{chio14})}
	Let $\varrho_\pm\in\mathbb{R}^+$ and $v_\pm\in\mathbb{R}^2$ be given constants, where $v_{-\,1}=v_{+\,1}$. 
	\begin{enumerate}
		\item If
		\begin{equation*}
		v_{+\,2} - v_{-\,2} \geq \int_0^{\varrho_-}\frac{\sqrt{p'(r)}}{r}\,\dd r + \int_0^{\varrho_+}\frac{\sqrt{p'(r)}}{r}\,\dd r,
		\end{equation*}
		then the standard solution to the problem \eqref{eq:euler}, \eqref{eq:ourinit} consists of a 1-rarefaction and a 3-rarefaction. The intermediate state $(\varrho_M,v_{M\,1},v_{M\,2})$ is a vacuum state, i.e. $\varrho_M=0$.
		\item If 
		\begin{equation*}
		\bigg|\int_{\varrho_-}^{\varrho_+}\frac{\sqrt{p'(r)}}{r}\,\dd r\bigg| < v_{+\,2} - v_{-\,2} < \int_0^{\varrho_-}\frac{\sqrt{p'(r)}}{r}\,\dd r + \int_0^{\varrho_+}\frac{\sqrt{p'(r)}}{r}\,\dd r,
		\end{equation*}
		then the standard solution to the problem \eqref{eq:euler}, \eqref{eq:ourinit} consists of a 1-rarefaction and a 3-rarefaction. The intermediate state $(\varrho_M,v_{M\,1},v_{M\,2})$ is given by 
		\begin{align*}
		\varrho_M&<\min\{\varrho_-,\varrho_+\}, \\
		v_{+\,2} - v_{-\,2} &= \int_{\varrho_M}^{\varrho_-}\frac{\sqrt{p'(r)}}{r}\,\dd r + \int_{\varrho_M}^{\varrho_+}\frac{\sqrt{p'(r)}}{r}\,\dd r, \\
		v_{M\,1} &= v_{-\,1} = v_{+\,1}, \\
		v_{M\,2} &= v_{-\,2} + \int_{\varrho_M}^{\varrho_-}\frac{\sqrt{p'(r)}}{r}\,\dd r.
		\end{align*}
		\item If
		\begin{equation*}
		\bigg|\int_{\varrho_-}^{\varrho_+}\frac{\sqrt{p'(r)}}{r}\,\dd r\bigg| = v_{+\,2} - v_{-\,2},
		\end{equation*}
		then the standard solution to the problem \eqref{eq:euler}, \eqref{eq:ourinit} consists of one rarefaction. More precisely this rarefaction is a 1-rarefaction if $\varrho_->\varrho_+$ and a 3-rarefaction if $\varrho_-<\varrho_+$.
		\item If $\varrho_->\varrho_+$ and
		\begin{equation*}
		-\sqrt{\frac{\big(\varrho_- - \varrho_+\big)\,\big(p(\varrho_-)-p(\varrho_+)\big)}{\varrho_-\,\varrho_+}} < v_{+\,2} - v_{-\,2} < \int_{\varrho_+}^{\varrho_-}\frac{\sqrt{p'(r)}}{r}\,\dd r,
		\end{equation*}
		then the standard solution to the problem \eqref{eq:euler}, \eqref{eq:ourinit} consists of a 1-rarefaction and a 3-shock. The intermediate state $(\varrho_M,v_{M\,1},v_{M\,2})$ is given by 
		\begin{align*}
		\varrho_+&<\varrho_M<\varrho_-, \\
		v_{+\,2} - v_{-\,2} &= \int_{\varrho_M}^{\varrho_-}\frac{\sqrt{p'(r)}}{r}\,\dd r - \sqrt{\frac{\big(\varrho_M - \varrho_+\big)\,\big(p(\varrho_M) - p(\varrho_+)\big)}{\varrho_M\,\varrho_+}}, \\
		v_{M\,1} &= v_{-\,1} = v_{+\,1}, \\
		v_{M\,2} &= v_{-\,2} + \int_{\varrho_M}^{\varrho_-}\frac{\sqrt{p'(r)}}{r}\,\dd r.
		\end{align*}
		\item If $\varrho_-<\varrho_+$ and
		\begin{equation*}
		-\sqrt{\frac{\big(\varrho_- - \varrho_+\big)\,\big(p(\varrho_-)-p(\varrho_+)\big)}{\varrho_-\,\varrho_+}} < v_{+\,2} - v_{-\,2} < \int_{\varrho_-}^{\varrho_+}\frac{\sqrt{p'(r)}}{r}\,\dd r,
		\end{equation*}
		then the standard solution to the problem \eqref{eq:euler}, \eqref{eq:ourinit} consists of a 1-shock and a 3-rarefaction. The intermediate state $(\varrho_M,v_{M\,1},v_{M\,2})$ is given by 
		\begin{align*}
		\varrho_-&<\varrho_M<\varrho_+, \\
		v_{+\,2} - v_{-\,2} &= 	\int_{\varrho_M}^{\varrho_+}\frac{\sqrt{p'(r)}}{r}\,\dd r - \sqrt{\frac{\big(\varrho_M - \varrho_-\big)\,\big(p(\varrho_M) - p(\varrho_-)\big)}{\varrho_M\,\varrho_-}}, \\
		v_{M\,1} &= v_{-\,1} = v_{+\,1}, \\
		v_{M\,2} &= v_{-\,2} - \sqrt{\frac{\big(\varrho_M - \varrho_-\big)\,\big(p(\varrho_M) - p(\varrho_-)\big)}{\varrho_M\,\varrho_-}}.
		\end{align*}
		\item If 
		\begin{equation*}
		v_{+\,2} - v_{-\,2} = -\sqrt{\frac{\big(\varrho_- - \varrho_+\big)\,\big(p(\varrho_-)-p(\varrho_+)\big)}{\varrho_-\,\varrho_+}},
		\end{equation*}
		then the standard solution to the problem \eqref{eq:euler}, \eqref{eq:ourinit} consists of one shock. More precisely this shock is a 1-shock if $\varrho_-<\varrho_+$ and a 3-shock if $\varrho_->\varrho_+$.
		\item If 
		\begin{equation*}
		v_{+\,2} - v_{-\,2} < -\sqrt{\frac{\big(\varrho_- - \varrho_+\big)\,\big(p(\varrho_-)-p(\varrho_+)\big)}{\varrho_-\,\varrho_+}},
		\end{equation*}
		then the standard solution to the problem \eqref{eq:euler}, \eqref{eq:ourinit} consists of a 1-shock and a 3-shock. The intermediate state $(\varrho_M,v_{M\,1},v_{M\,2})$ is given by 
		\begin{align*}
		\varrho_M&>\max\{\varrho_-,\varrho_+\}, \\
		v_{+\,2} - v_{-\,2} &= - \sqrt{\frac{\big(\varrho_M - \varrho_+\big)\,\big(p(\varrho_M) - p(\varrho_+)\big)}{\varrho_M\,\varrho_+}} - \sqrt{\frac{\big(\varrho_M - \varrho_-\big)\,\big(p(\varrho_M) - p(\varrho_-)\big)}{\varrho_M\,\varrho_-}}, \\
		v_{M\,1} &= v_{-\,1} = v_{+\,1}, \\
		v_{M\,2} &= v_{-\,2} - \sqrt{\frac{\big(\varrho_M - \varrho_-\big)\,\big(p(\varrho_M) - p(\varrho_-)\big)}{\varrho_M\,\varrho_-}}.
		\end{align*}
	\end{enumerate}
	In each case the standard solution is admissible.
	\label{prop:standardsolution}
\end{proposition}

\begin{proof}
	We don't want to present the whole proof here, but we are going to say few words about it. We have to solve the one-dimensional Riemann problem
	\begin{equation}
	\begin{split}
	\partial_t \varrho + \partial_{x_2}(\varrho\,v_2) &= 0, \\
	\partial_t (\varrho\,v_1) + \partial_{x_2}(\varrho\,v_1\,v_2) &= 0, \\
	\partial_t (\varrho\,v_2) + \partial_{x_2}\big(\varrho\,v_2^2 + p(\varrho)\big) &= 0,
	\end{split}
	\label{eq:1deuler}
	\end{equation}
	\begin{equation}
	\begin{split}
	\varrho(0,x_2) &= \left\{
	\begin{array}[c]{ll}
	\varrho_- & \text{ if }x_2<0 \\
	\varrho_+ & \text{ if }x_2>0
	\end{array}
	\right. ,\\
	v(0,x_2) &= \left\{
	\begin{array}[c]{ll}
	v_- & \text{ if }x_2<0 \\
	v_+ & \text{ if }x_2>0
	\end{array}
	\right. ,
	\end{split}
	\label{eq:1dinit}
	\end{equation}
	where the unknowns $\varrho=\varrho(t,x_2)\in\mathbb{R}^+$ and $v=v(t,x_2)\in\mathbb{R}^2$ are now functions of the time $t\in[0,\infty)$ and the position $x_2\in\mathbb{R}$. Additionally we want the following admissibility condition to be true
	\begin{equation}
	\partial_t\bigg(\varrho\,\varepsilon(\varrho) + \varrho\,\frac{|v|^2}{2}\bigg) + \partial_{x_2}\Bigg(\bigg(\varrho\,\varepsilon(\varrho) + \varrho\,\frac{|v|^2}{2} + p(\varrho)\bigg)v_2\Bigg) \leq 0.
	\label{eq:1dadm}
	\end{equation}
	
	It is well-known that there exists a weak solution to \eqref{eq:1deuler}, \eqref{eq:1dinit}, \eqref{eq:1dadm} which consists of shocks, rarefactions and contact discontinuities. By well-known methods (see textbooks, e.g. the ones by C. Dafermos \cite[Chapters 7 - 9]{dafermos} or R. LeVeque \cite[Chapters 13, 14]{leveque}) one can compute this solution and one ends up with the seven cases in proposition \ref{prop:standardsolution}. Parts of proposition \ref{prop:standardsolution} together with a proof can be found in \cite[Lemma 2.4]{chio14}, too. \qed
\end{proof}

The aim is to check if the standard solution is unique or if there are other admissible weak solutions. This question on uniqueness concerning admissible weak solutions to problem \eqref{eq:euler}, \eqref{eq:ourinit} has been discussed in previous papers. The results are summarized in the following theorem.

\begin{theorem}
	Let $\varrho_\pm\in\mathbb{R}^+$ and $v_\pm\in\mathbb{R}^2$ be given constants, where $v_{-\,1}=v_{+\,1}$. The following table summarizes the results on uniqueness of admissible weak solutions. In the cases where the solution is not unique, there are even infinitely many admissible weak solutions.
	\renewcommand{\arraystretch}{1.5}
	\begin{table}[h] 
		\caption{Results on uniqueness of the standard solution to problem \eqref{eq:euler}, \eqref{eq:ourinit}}
		\centering
		\begin{tabular}{|m{0.18\linewidth}|m{0.16\linewidth}|m{0.17\linewidth}|m{0.19\linewidth}|} \hline
			\centering Standard solution consists of & \centering Case in proposition \ref{prop:standardsolution} & \centering Solution unique? & \centering Reference \tabularnewline \hline \hline
			\centering two rarefactions with vacuum & \centering 1 & \centering yes & \centering \cite{chen} \tabularnewline \hline 
			\centering two rarefactions, no vacuum & \centering 2 & \centering yes & \centering \cite{chen}, \cite{feir15} \tabularnewline \hline
			\centering one rarefaction & \centering 3 & \centering yes & \centering \cite{chen}, \cite{feir15} \tabularnewline \hline
			\centering 1-rarefaction, 3-shock & \centering 4 & \centering no & \centering Theorem \ref{thm:SRexistence}, one example in \cite{chio15} \tabularnewline \hline
			\centering 1-shock, 3-rarefaction & \centering 5 & \centering no & \centering Theorem \ref{thm:SRexistence}, one example in \cite{chio15} \tabularnewline \hline
			\centering one shock & \centering 6 & \centering no & \centering Theorem \ref{thm:Sexistence} \tabularnewline \hline
			\centering two shocks & \centering 7 & \centering no & \centering \cite{chio14} \tabularnewline \hline
		\end{tabular}
	\end{table}
	\renewcommand{\arraystretch}{1}
	\label{thm:summaryriemann}
\end{theorem}

For an exact proof we refer to the given references. What we want to do here is to describe the basic ideas of the papers cited in the table above. 

If the standard solution is continuous, i.e. it consists only of rarefactions, it is unique. To prove this uniqueness G.-Q. Chen and J. Chen \cite{chen}, and independently E. Feireisl and O. Kreml \cite{feir15} use a relative entropy inequality. 

If the standard solution consists of two shocks, E. Chiodaroli and O. Kreml \cite{chio14} showed that there are infinitely many other admissible weak solutions. In other words the standard solution is non-unique in this case. To prove this they apply the method of convex integration, which was developed by C. De Lellis and L. Sz{\'{e}}kelyhidi \cite{dls09}, \cite{dls10} and leads to infinitely many admissible weak solutions, called \emph{wild solutions}. 

Similar techniques are used by E. Chiodaroli, C. De Lellis and O. Kreml \cite{chio15} to show that for one particular example of initial states, to which the standard solution consists of one shock and one rarefaction, there are infinitely many other admissible weak solutions. 

The cases where the initial data is such that the standard solution consists of just one shock or one shock and one rarefaction (apart from the particular example in \cite{chio15}) remain open and will be covered in this paper. To show non-uniqueness we will use the same strategy as in \cite{chio14} and \cite{chio15}, where the crucial point is to work with an auxiliary state.

\begin{remark}
	We want to add some words on the pressure laws used in the above references.
	\begin{itemize}
		\item G.-Q. Chen and J. Chen \cite{chen} write that their results hold for the same pressure law as we consider but with $\gamma>1$. However it is possible to use $\gamma=1$ since they only consider non-strict inequalities.
		\item E. Feireisl's and O. Kreml's \cite{feir15} results hold for any convex, strictly increasing $C^1$-pressure function. Hence for our pressure law, too.
		\item E. Chiodaroli and O. Kreml \cite{chio14} use our pressure law with $K=1$. However their results are true for any $K>0$. 
	\end{itemize} 
\end{remark}

\begin{remark}
	The case $v_{-\,1}\neq v_{+\,1}$ is not considered in this paper. First results on the question on uniqueness of admissible weak solutions in this case can be found in \cite{brezina}.
\end{remark}


\section{A sufficient condition for non-uniqueness}

This section is a summary of results by E. Chiodaroli, C. De Lellis and O. Kreml \cite{chio14}, \cite{chio15}, which are used to show non-uniqueness. We will use their results in this paper, too. We choose to cite \cite{chio14}, but the same definitions can be found in \cite{chio15}, too.

\subsection{Definitions}

\begin{definition} \emph{(fan partition, see \cite[Definition 4]{chio14})}
	Let \mbox{$\mu_0<\mu_1$} real numbers. A fan partition of $(0,\infty)\times\mathbb{R}^2$ consists of three open sets $P_-,P_1,P_+$ of the form
	\begin{align*}
	P_-&=\{(t,x):t>0\text{ and }x_2<\mu_0\,t\}, \\
	P_1&=\{(t,x):t>0\text{ and }\mu_{0}\,t<x_2<\mu_1\,t\}, \\
	P_+&=\{(t,x):t>0\text{ and }x_2>\mu_1\,t\},
	\end{align*}
	see figure \ref{fig:fanpart}.
	\label{defn:fanpart}
\end{definition}

\begin{figure}[hbt]
	\centering
	\includegraphics[width=0.85\textwidth]{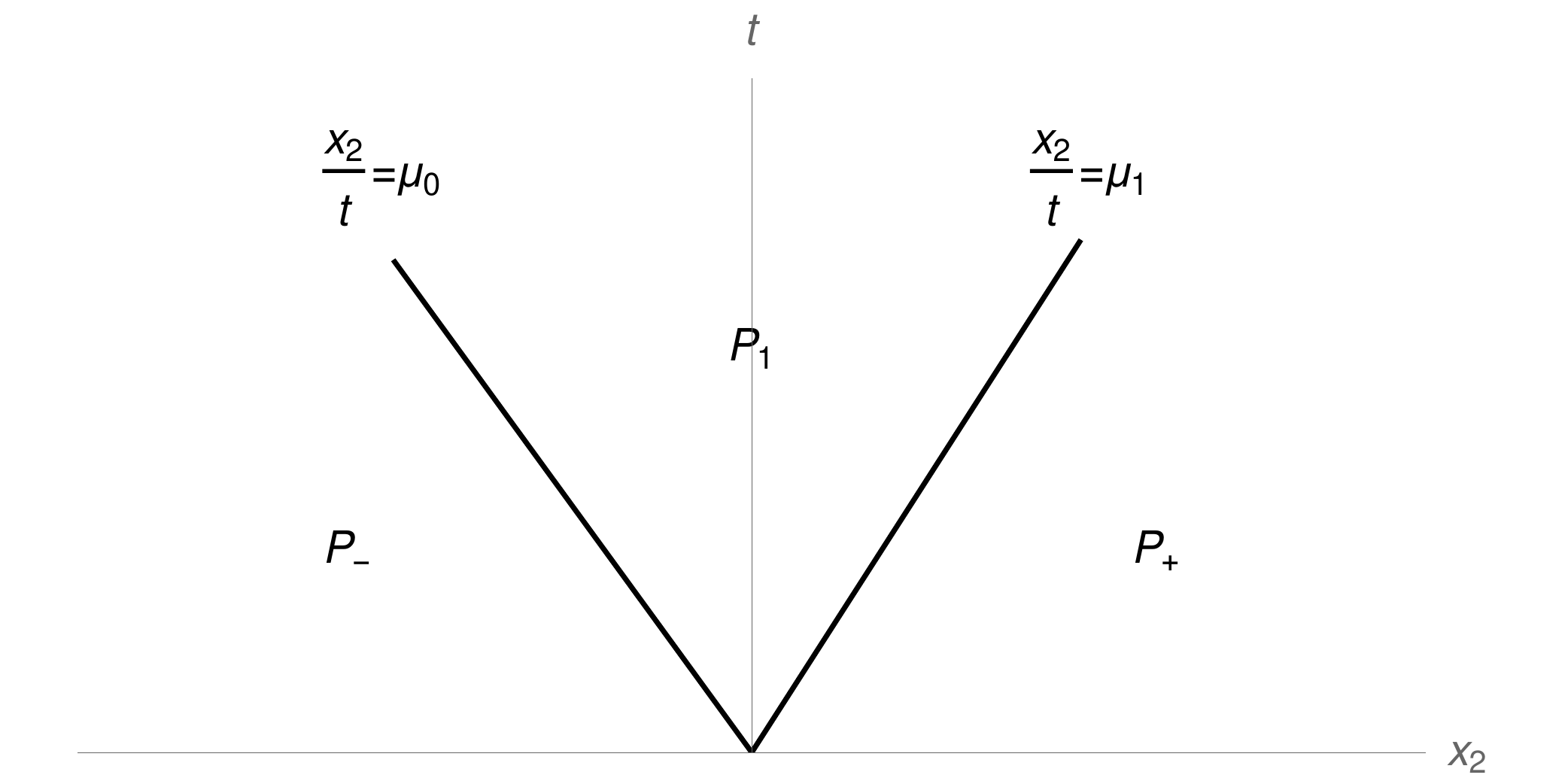}
	\caption{fan partition}
	\label{fig:fanpart}
\end{figure} 

We need to introduce the following notation. The set of real $2\times 2$ matrices which are symmetric will be denoted as $\mathcal{S}^{2\times 2}$, whose subset of symmetric traceless matrices is called $\mathcal{S}_0^{2\times 2}$. In addition to that we write $\ \text{Id}\ $ for the $2\times 2$ identity matrix and $\mathbf{1}_P$ for the indicater function on $P$.

\begin{definition} \emph{(admissible fan subsolution, see \cite[Definitions 5 and 6]{chio14})}
	An admissible fan subsolution to the Euler system \eqref{eq:euler} with initial condition \eqref{eq:ourinit} is a triple $(\overline{\varrho},\overline{v},\overline{u}):(0,\infty)\times\mathbb{R}^2\rightarrow(\mathbb{R}^+\times\mathbb{R}^2\times\mathcal{S}_0^{2\times2})$ of piecewise constant functions, which satisfies the following properties:
	\begin{enumerate}
		\item There exists a fan partition of $(0,\infty)\times\mathbb{R}^2$ and constants $\varrho_1\in\mathbb{R}^+$, $v_1\in\mathbb{R}^2$ and $u_1\in\mathcal{S}_0^{2\times 2}$, such that
		\begin{align*}
		(\overline{\varrho},\overline{v},\overline{u})&=\sum\limits_{i\in\{-,+\}} \bigg(\varrho_i\,,\,v_i\,,\,v_i\otimes v_i - \frac{|v_i|^2}{2}\,\text{Id}\bigg)\,\mathbf{1}_{P_i} + (\varrho_1\,,\,v_1\,,\,u_1)\,\mathbf{1}_{P_1},
		\end{align*}
		where $\varrho_{\pm},v_{\pm}$ are constants given by the initial condition \eqref{eq:ourinit}. 
		\item There is a constant $C_1\in\mathbb{R}^+$ such that\footnote{Here we have an inequality of matrices, which is meant in the sense of definiteness. That means, that $A<B$ for $A,B\in\mathcal{S}^{2\times 2}$, if $B-A$ is positive definite.} 
		\begin{equation*}
		v_1\otimes v_1-u_1 < \frac{C_1}{2}\,\text{Id}.
		\end{equation*}
		\item For all test functions $(\psi,\phi)\in C_c^\infty([0,\infty)\times\mathbb{R}^2,\mathbb{R}\times\mathbb{R}^2)$ the following identities hold:
		\begin{align*}
		\int_0^\infty\int_{\mathbb{R}^2}\big(\overline{\varrho}\,\partial_t\psi + \overline{\varrho}\,\overline{v}\cdot\nabla_x\psi\big)\dd x\,\dd t + \int_{\mathbb{R}^2} \varrho_0(x)\,\psi(0,x)\,\dd x\ &=\ 0, \\
		\int_0^\infty\int_{\mathbb{R}^2}\Bigg[\overline{\varrho}\,\overline{v}\cdot\partial_t\phi + \overline{\varrho}\,\Big(\big(\overline{v}\otimes \overline{v}\big)\,\mathbf{1}_{P_- \cup P_+} + u_1\,\mathbf{1}_{P_1}\Big):D_x\phi\qquad\qquad\qquad&  \\
		+ \bigg(p(\overline{\varrho}) +\frac{1}{2}\,\varrho_1\,C_1\,\mathbf{1}_{P_1}\bigg)\,\text{div}_x\phi\Bigg]\dd x\,\dd t  + \int_{\mathbb{R}^2} \varrho_0(x)\,v_0(x)\cdot\phi(0,x)\,\dd x\  &=\ 0. 
		\end{align*}
		\item For every non-negative test function $\varphi\in C_c^\infty([0,\infty)\times\mathbb{R}^2,\mathbb{R}_0^+)$ the inequality
		\begin{align*}
		&\int_0^\infty\int_{\mathbb{R}^2}\Bigg[\bigg(\overline{\varrho}\,\varepsilon(\overline{\varrho}) + \frac{1}{2}\,\overline{\varrho}\,\Big(|\overline{v}|^2\,\mathbf{1}_{P_- \cup P_+} + C_1\,\mathbf{1}_{P_1}\Big)\bigg)\,\partial_t\varphi \\
		&\qquad+ \bigg(\overline{\varrho}\,\varepsilon(\overline{\varrho})+p(\overline{\varrho}) + \frac{1}{2}\,\overline{\varrho}\,\Big(|\overline{v}|^2\,\mathbf{1}_{P_- \cup P_+} + C_1\,\mathbf{1}_{P_1}\Big)\bigg)\,\overline{v}\cdot\nabla_x\varphi\Bigg]\dd x\,\dd t \\
		&\quad+ \int_{\mathbb{R}^2} \varrho_0(x)\,\bigg(\varepsilon(\varrho_0(x)) + \frac{|v_0(x)|^2}{2}\bigg)\,\varphi(0,x)\,\dd x \quad\geq \quad 0 
		\end{align*}
		is fulfilled.
	\end{enumerate}
	\label{defn:admfansubs}
\end{definition}

\subsection{The condition}
It turns out that the existence of an admissible fan subsolution implies existence of infinitely many admissible weak solutions.
\begin{theorem} \emph{(see \cite[Proposition 3.1]{chio14})}
	Let $(\varrho_\pm,v_\pm)$ be such that there exists an admissible fan subsolution $(\overline{\varrho},\overline{v},\overline{u})$ to the Cauchy problem \eqref{eq:euler}, \eqref{eq:ourinit}. Then there are infinitely many admissible weak solutions $(\varrho,v)$ to \eqref{eq:euler}, \eqref{eq:ourinit} with the following properties:
	\begin{itemize}
		\item $\varrho=\overline{\varrho}$,
		\item $v(t,x)=\overline{v}(t,x)$ for almost all $(t,x)\in P_- \cup P_+$, 
		\item $|v(t,x)|^2=C_1$ for almost all $(t,x)\in P_1$.
	\end{itemize}
	\label{thm:condition}
\end{theorem}

For the proof we refer to \cite{chio14}. 


\section{The algebraic equations}

Because of theorem \ref{thm:condition} it suffices to show existence of an admissible fan subsolution in order to prove existence of infinitely many admissible weak solutions. In order to construct an admissible fan subsolution we translate definition \ref{defn:admfansubs} into a system of algebraic equations and inequalities for a set of unknown values. The following propositions can be found both in \cite{chio14} and \cite{chio15}.

\begin{proposition} \emph{(see \cite[Proposition 4.1]{chio14})}
	Let $\varrho_-,\varrho_+\in\mathbb{R}^+$, $v_-,v_+\in\mathbb{R}^2$ be given (see initial condition \eqref{eq:ourinit}). The constants $\mu_0,\mu_1\in\mathbb{R}$, $\varrho_1\in\mathbb{R}^+$, 
	\begin{align*}
	v_1&=\left(
	\begin{array}[c]{l}
	v_{1\,1}\\
	v_{1\,2}
	\end{array}\right) \in\mathbb{R}^2, &u_1&=\left(
	\begin{array}[c]{rr}
	u_{1\,11} & u_{1\,12} \\
	u_{1\,12} & -u_{1\,11}
	\end{array} \right)\in\mathcal{S}_0^{2\times 2}
	\end{align*}
	and $C_1\in\mathbb{R}^+$ define an admissible fan subsolution to the Cauchy problem \eqref{eq:euler}, \eqref{eq:ourinit} if and only if they fulfill the following algebraic equations and inequalities: 
	\begin{itemize}
		\item Order of the speeds:
		\begin{equation}
		\mu_0<\mu_1
		\label{eq:1order}
		\end{equation}
		\item Rankine Hugoniot conditions on the left interface:
		\begin{align}
		\mu_0\,(\varrho_- - \varrho_1) &= \varrho_-\,v_{-\,2} - \varrho_1\,v_{1\,2} \label{eq:1rhl1}\\
		\mu_0\,(\varrho_-\,v_{-\,1} - \varrho_1\,v_{1\,1}) &= \varrho_-\,v_{-\,1}\,v_{-\,2} - \varrho_1\,u_{1\,12} \label{eq:1rhl2}\\
		\mu_0\,(\varrho_-\,v_{-\,2} - \varrho_1\,v_{1\,2}) &= \varrho_-\,v_{-\,2}^2 + \varrho_1\,u_{1\,11} + p(\varrho_-) - p(\varrho_1) - \varrho_1\,\frac{C_1}{2} \label{eq:1rhl3}
		\end{align}
		\item Rankine Hugoniot conditions on the right interface:
		\begin{align}
		\mu_1\,(\varrho_1 - \varrho_+) &= \varrho_1\,v_{1\,2} - \varrho_+\,v_{+\,2} \label{eq:1rhr1}\\
		\mu_1\,(\varrho_1\,v_{1\,1} - \varrho_+\,v_{+\,1}) &= \varrho_1\,u_{1\,12} - \varrho_+\,v_{+\,1}\,v_{+\,2} \label{eq:1rhr2}\\
		\mu_1\,(\varrho_1\,v_{1\,2} - \varrho_+\,v_{+\,2}) &= - \varrho_1\,u_{1\,11} - \varrho_+\,v_{+\,2}^2 + p(\varrho_1) - p(\varrho_+) + \varrho_1\,\frac{C_1}{2} \label{eq:1rhr3}
		\end{align}
		\item Subsolution condition:
		\begin{align}
		v_{1\,1}^2 + v_{1\,2}^2 &< C_1 \label{eq:1sc1}\\
		\bigg(\frac{C_1}{2} - v_{1\,1}^2 + u_{1\,11}\bigg) \bigg(\frac{C_1}{2} - v_{1\,2}^2 - u_{1\,11}\bigg) - (u_{1\,12}-v_{1\,1}\,v_{1\,2})^2 &> 0 \label{eq:1sc2}
		\end{align}
		\item Admissibility condition on the left interface:
		\begin{align}
		&\mu_0\,\bigg(\varrho_-\,\varepsilon(\varrho_-)+\varrho_-\,\frac{|v_-|^2}{2}-\varrho_1\,\varepsilon(\varrho_1)-\varrho_1\,\frac{C_1}{2}\bigg) \notag\\
		&\leq\big(\varrho_-\,\varepsilon(\varrho_-)+p(\varrho_-)\big)\,v_{-\,2}-\big(\varrho_1\,\varepsilon(\varrho_1)+p(\varrho_1)\big)\,v_{1\,2}+\varrho_-\,v_{-\,2}\,\frac{|v_-|^2}{2}-\varrho_1\,v_{1\,2}\,\frac{C_1}{2}
		\label{eq:1adml}
		\end{align}
		\item Admissibility condition on the right interface:
		\begin{align}
		&\mu_1\,\bigg(\varrho_1\,\varepsilon(\varrho_1)+\varrho_1\,\frac{C_1}{2}-\varrho_+\,\varepsilon(\varrho_+)-\varrho_+\,\frac{|v_+|^2}{2}\bigg) \notag\\
		&\leq\big(\varrho_1\,\varepsilon(\varrho_1)+p(\varrho_1)\big)\,v_{1\,2}-\big(\varrho_+\,\varepsilon(\varrho_+)+p(\varrho_+)\big)\,v_{+\,2}+\varrho_1\,v_{1\,2}\,\frac{C_1}{2}-\varrho_+\,v_{+\,2}\,\frac{|v_+|^2}{2}
		\label{eq:1admr}
		\end{align}
	\end{itemize}
	\label{prop:1alggl}
\end{proposition}

\begin{remark}
	The above proposition \ref{prop:1alggl} holds even if $v_{-\,1}\neq v_{+\,1}$.
\end{remark}

The equations and inequalities in proposition \ref{prop:1alggl} can be simplified further if $v_{-\,1}=v_{+\,1}$, which is the content of the following proposition.

\begin{proposition} \emph{(see \cite[Lemma 4.4]{chio14})}
	Let $\varrho_-,\varrho_+\in\mathbb{R}^+$, $v_-,v_+\in\mathbb{R}^2$ with $v_{-\,1}=v_{+\,1}$ be given (see initial condition \eqref{eq:ourinit}). There exists an admissible fan subsolution to the Cauchy problem \eqref{eq:euler}, \eqref{eq:ourinit} if and only if there exist constants $\mu_0,\mu_1\in\mathbb{R}$, $\varrho_1\in\mathbb{R}^+$, $v_{1\,2}\in\mathbb{R}$ and $\delta_1,\delta_2\in\mathbb{R}$ such that the following algebraic equations and inequalities hold:
	\begin{itemize}
		\item Order of the speeds:
		\begin{equation}
		\mu_0<\mu_1
		\label{eq:2order}
		\end{equation}
		\item Rankine Hugoniot conditions on the left interface
		\begin{align}
		\mu_0\,(\varrho_- - \varrho_1) &= \varrho_-\,v_{-\,2} - \varrho_1\,v_{1\,2} \label{eq:2rhl1}\\
		\mu_0\,(\varrho_-\,v_{-\,2} - \varrho_1\,v_{1\,2}) &= \varrho_-\,v_{-\,2}^2 - \varrho_1\,(v_{1\,2}^2+\delta_1) + p(\varrho_-) - p(\varrho_1) \label{eq:2rhl2}
		\end{align}
		\item Rankine Hugoniot conditions on the right interface
		\begin{align}
		\mu_1\,(\varrho_1 - \varrho_+) &= \varrho_1\,v_{1\,2} - \varrho_+\,v_{+\,2} \label{eq:2rhr1}\\
		\mu_1\,(\varrho_1\,v_{1\,2} - \varrho_+\,v_{+\,2}) &= \varrho_1\,(v_{1\,2}^2+\delta_1) - \varrho_+\,v_{+\,2}^2 + p(\varrho_1) - p(\varrho_+) \label{eq:2rhr2}
		\end{align}
		\item Subsolution condition
		\begin{align}
		\delta_1>0 \label{eq:2sc1}\\
		\delta_2>0 \label{eq:2sc2}
		\end{align}
		\item Admissibility condition on the left interface
		\begin{align}
		(v_{1\,2}-v_{-\,2})\,&\bigg(p(\varrho_-)+p(\varrho_1)-2\,\varrho_-\,\varrho_1\,\frac{\varepsilon(\varrho_-)-\varepsilon(\varrho_1)}{\varrho_- -\varrho_1}\bigg) \notag\\
		&\leq\delta_1\,\varrho_1\,(v_{1\,2}+v_{-\,2})-(\delta_1+\delta_2)\,\frac{\varrho_-\,\varrho_1\,(v_{1\,2}-v_{-\,2})}{\varrho_- -\varrho_1}
		\label{eq:2adml}
		\end{align}
		\item Admissibility condition on the right interface
		\begin{align}
		(v_{+\,2}-v_{1\,2})\,&\bigg(p(\varrho_1)+p(\varrho_+)-2\,\varrho_1\,\varrho_+\,\frac{\varepsilon(\varrho_1)-\varepsilon(\varrho_+)}{\varrho_1 -\varrho_+}\bigg) \notag\\
		&\leq-\delta_1\,\varrho_1\,(v_{+\,2}+v_{1\,2})+(\delta_1+\delta_2)\,\frac{\varrho_1\,\varrho_+\,(v_{+\,2}-v_{1\,2})}{\varrho_1 -\varrho_+}
		\label{eq:2admr}
		\end{align}
	\end{itemize}
	\label{prop:2alggl}
\end{proposition}


\section{Lemmas}
Later on we will also need the following lemmas.
\begin{lemma} \emph{(see \cite[Lemma 2.1]{chio14})}
	For all $\varrho_-\neq\varrho_+$, $\varrho_\pm>0$ it holds that 
	\begin{equation}
	p(\varrho_-) + p(\varrho_+) - 2\,\varrho_-\,\varrho_+\,\frac{\varepsilon(\varrho_+) - \varepsilon(\varrho_-)}{\varrho_+ - \varrho_-} > 0.
	\label{eq:Lem1}
	\end{equation}
	\label{lemma:auschiokreml}
\end{lemma}

\begin{proof}
	The lemma is proved by E. Chiodaroli and O. Kreml \cite[Lemma 2.1]{chio14}. They show the result for the pressure law $p(\varrho)=\varrho^\gamma$, where $\gamma\geq 1$. In other words $K=1$. However \eqref{eq:Lem1} is also true for the more general pressure law $p(\varrho)=K\,\varrho^\gamma$, where $K>0$ and $\gamma\geq 1$. \qed
\end{proof}

\begin{lemma}
	For all $\varrho_- < \varrho_+$ the following inequality is fulfilled:
	\begin{equation}
	\int_{\varrho_-}^{\varrho_+}\frac{\sqrt{p'(r)}}{r}\,\dd r < \sqrt{\frac{(\varrho_- - \varrho_+)\,\big(p(\varrho_-)-p(\varrho_+)\big)}{\varrho_-\,\varrho_+}}.
	\label{eq:Lem2}
	\end{equation} 
	\label{lemma:introot}
\end{lemma}

\begin{proof} 
	First we consider the case $\gamma>1$. In this case the integral can be computed to:
	\begin{equation*}
	\int_{\varrho_-}^{\varrho_+}\frac{\sqrt{p'(r)}}{r}\,\dd r =  \frac{2}{\gamma-1}\,\Big(\sqrt{p'(\varrho_+)} - \sqrt{p'(\varrho_-)}\Big).
	\end{equation*} 
	
	Hence the equation \eqref{eq:Lem2} turns into 
	\begin{equation*}
	\frac{2}{\gamma-1}\,\Big(\sqrt{p'(\varrho_+)} - \sqrt{p'(\varrho_-)}\Big) < \sqrt{\frac{(\varrho_- - \varrho_+)\,\big(p(\varrho_-)-p(\varrho_+)\big)}{\varrho_-\,\varrho_+}}.
	\end{equation*}
	
	Because $p''(\varrho)>0$ for all $\varrho>0$ and $\gamma>1$, $p'$ is increasing. Hence both sides of the above inequality are positive and therefore it is equivalent to 
	\begin{equation*}
	\frac{4}{(\gamma-1)^2}\,\Big(\sqrt{p'(\varrho_+)} - \sqrt{p'(\varrho_-)}\Big)^2 < \frac{(\varrho_- - \varrho_+)\,\big(p(\varrho_-)-p(\varrho_+)\big)}{\varrho_-\,\varrho_+}.
	\end{equation*}
	
	Remember that $p(\varrho)=K\,\varrho^\gamma$ and $p'(\varrho)=K\,\gamma\,\varrho^{\gamma-1}$. Divide the inequality above by $K$ and $\varrho_-^{\gamma-1}$, and define $z:=\frac{\varrho_+}{\varrho_-}$:
	\begin{equation*}
	\frac{4\,\gamma}{(\gamma-1)^2}\,\big(z^{\gamma-1} - 2\,z^{\frac{\gamma-1}{2}} + 1\big) < \frac{1}{z}\,(z-1)\,(z^\gamma-1).
	\end{equation*}
	
	Let 
	\begin{equation*}
	f(z) := (z-1)\,(z^\gamma-1) - \frac{4\,\gamma}{(\gamma-1)^2}\,\big(z^{\gamma} - 2\,z^{\frac{\gamma+1}{2}} + z\big),
	\end{equation*}
	then it is sufficient to prove that $f(z)>0$ for all $z>1$. It is easy to recalculate that 
	\begin{align*}
	f'(z) &= (z^\gamma-1) + (z-1)\,\gamma\,z^{\gamma-1} - \frac{4\,\gamma}{(\gamma-1)^2}\,\big(\gamma\,z^{\gamma-1} - 	(\gamma+1)\,z^{\frac{\gamma-1}{2}} + 1\big), \\
	f''(z) &= \gamma\,z^{\gamma-1} + (z-1)\,\gamma\,(\gamma-1)\,z^{\gamma-2} + \gamma\,z^{\gamma-1} \\
	&\qquad- \frac{4\,\gamma}{(\gamma-1)^2}\,\Big(\gamma\,(\gamma-1)\,z^{\gamma-2} - (\gamma+1)\,\frac{\gamma-1}{2}\,z^{\frac{\gamma-3}{2}}\Big) \\
	&= \underbrace{\vphantom{\bigg(\bigg)}\gamma\,(\gamma+1)\,z^{\frac{\gamma-3}{2}}}_{>0}\,\underbrace{\bigg[z^{\frac{\gamma-1}{2}} \bigg( z - \frac{\gamma+1}{\gamma-1} \bigg) + \frac{2}{\gamma-1}\bigg]}_{=:g(z)}.
	\end{align*}
	
	Finally 
	\begin{align*}
	g'(z)&=\frac{\gamma-1}{2}\,z^{\frac{\gamma-3}{2}}\,\bigg( z - \frac{\gamma+1}{\gamma-1} \bigg) + z^{\frac{\gamma-1}{2}} \\
	&= \frac{\gamma+1}{2}\,z^{\frac{\gamma-3}{2}}\,\big( z - 1\big) \quad >\quad 0,
	\end{align*}
	what implies with $g(1)=0$ that $g(z)>0$ for $z>1$. Hence $f''(z)>0$ and with $f'(1)=f(1)=0$ we obtain the wanted property $f(z)>0$ for all $z>1$.
	
	It remains to consider the case $\gamma=1$. Here we have to show that 
	\begin{equation*}
	\text{log}\bigg(\frac{\varrho_+}{\varrho_-}\bigg) < \sqrt{\frac{\varrho_+}{\varrho_-}} - \sqrt{\frac{\varrho_-}{\varrho_+}}.
	\end{equation*}
	This inequality can be proved by similar methods, which we leave to the reader. \qed
\end{proof}

\begin{lemma}
	For all $\varrho_- < \varrho_M < \varrho_+$ the following inequality is fulfilled:
	\begin{equation}
	\sqrt{\frac{\big(\varrho_M - \varrho_-\big)\,\big(p(\varrho_M)-p(\varrho_-)\big)}{\varrho_-\,\varrho_M}} < \sqrt{\frac{\big(\varrho_+ - \varrho_-\big)\,\big(p(\varrho_+)-p(\varrho_-)\big)}{\varrho_-\,\varrho_+}}
	\label{eq:Lem3}
	\end{equation} 
	\label{lemma:rootroot}
\end{lemma}

\begin{proof}
	It suffices to show that
	\begin{equation*}
	\frac{\big(\varrho_M - \varrho_-\big)\,\big(p(\varrho_M) - p(\varrho_-)\big)}{\varrho_M\,\varrho_-} < \frac{\big(\varrho_2 - \varrho_-\big)\,\big(p(\varrho_2) - p(\varrho_-)\big)}{\varrho_2\,\varrho_-}
	\end{equation*}
	which is equivalent to
	\begin{equation*}
	\bigg(\frac{1}{\varrho_-} - \frac{1}{\varrho_M}\bigg)\,\big(p(\varrho_M) - p(\varrho_-)\big) < \bigg(\frac{1}{\varrho_-} - \frac{1}{\varrho_2}\bigg)\,\big(p(\varrho_2) - p(\varrho_-)\big).
	\end{equation*}
	
	Since $\varrho_-<\varrho_M<\varrho_2$ and $p$ strictly increasing we obtain 
	\begin{equation*}
	0 < p(\varrho_M) - p(\varrho_-) < p(\varrho_2) - p(\varrho_-) \qquad\text{ and }\qquad 0 < \frac{1}{\varrho_-} - \frac{1}{\varrho_M} < \frac{1}{\varrho_-} - \frac{1}{\varrho_2}
	\end{equation*}
	and therefore the desired inequality \eqref{eq:Lem3}. \qed
\end{proof}


\section{The standard solution consists of a shock and a rarefaction}
\label{section:SR}

Now we are ready to begin with the main part of this paper. 

\begin{remark}
	Because of the rotational invariance of the Euler system, it is enough to consider the case where the standard solution consists of a 1-shock and a 3-rarefaction. If it is the other way round, we just rotate the coordinate system 180 degrees to obtain a new initial data
	\begin{equation}
	\begin{split}
	(\varrho_{-\,\text{new}},v_{-\,\text{new}}) &= (\varrho_+,-v_+) \\
	(\varrho_{+\,\text{new}},v_{+\,\text{new}}) &= (\varrho_-,-v_-).
	\end{split}
	\label{eq:rotinitSR}
	\end{equation}
	Note that the sign of the velocities changes during this transformation. In view of proposition \ref{prop:standardsolution} it is easy to check that the standard solution to the problem with rotated initial data \eqref{eq:rotinitSR} consists of a 1-shock and a 3-rarefaction.
\end{remark} 

Let the initial values $\varrho_\pm\in\mathbb{R}^+$ and $v_\pm\in\mathbb{R}^2$ be such that the standard solution consists of a 1-shock and a 3-rarefaction. By proposition \ref{prop:standardsolution} this means, that 
\begin{equation}
\begin{split}
\varrho_-&<\varrho_+\quad\text{ and}\\
-\sqrt{\frac{(\varrho_- - \varrho_+)\,\big(p(\varrho_-)-p(\varrho_+)\big)}{\varrho_-\,\varrho_+}} &< v_{+\,2} - v_{-\,2} < \int_{\varrho_-}^{\varrho_+}\frac{\sqrt{p'(r)}}{r}\,\dd r.
\end{split}
\label{eq:1S3R}
\end{equation}

\subsection{Existence of admissible fan subsolutions}

It was shown by Chiodaroli, De Lellis and Kreml \cite{chio15} that for one explicit example there exists an admissible fan subsolution and hence infinitely many admissibel weak solutions. Unfortunately there exist other examples where there are no admissible fan subsolutions, in other words, where we can not simply introduce a wedge with wild solutions. However we won't prove this here, since we want to show existence of infinitely many admissible weak solutions for \emph{all} examples of initial states that fulfill \eqref{eq:1S3R}. To achieve this, we need to slightly modify the approach in \cite{chio15}. 

First of all we want to find a criterion which tells us whether an admissible fan subsolution to given initial states, that fulfill \eqref{eq:1S3R}, exists or not. In order to do this we will rearrange the equations and inequalities in proposition \ref{prop:2alggl}. The latter proposition says that we have to find six real numbers that fulfill a set of four equations and five inequalities. As in \cite{chio14} the idea is now to choose two parameters and try to express the other four values as functions of these parameters, since there are four equations available. Because $\delta_2$ doesn't appear in equations \eqref{eq:2rhl1} - \eqref{eq:2rhr2}, it is a good choice to take $\delta_2$ as one parameter. We set $\rho_1$ to be the other parameter. We will be able to express $\mu_0$, $\mu_1$, $v_{1\,2}$ and $\delta_1$ as functions of $\rho_1$.

\begin{theorem}
	There exists an admissible fan subsolution to the Cauchy problem \eqref{eq:euler}, \eqref{eq:ourinit} if and only if there exist constants $\varrho_1,\delta_2\in\mathbb{R}^+$ that fulfill
	\begin{align} 
	\varrho_- &< \varrho_1 < \varrho_+, \label{eq:SRcond1} \\[0.3cm]
	\delta_1^\star(\varrho_1) &>0, \label{eq:SRcond2} 
	\end{align} 
	\vspace{-0.7cm}
	\begin{align}
	\begin{split}
	&(v_{1\,2}^\star(\varrho_1)-v_{-\,2})\,\bigg(p(\varrho_-)+p(\varrho_1)-2\,\varrho_-\,\varrho_1\,\frac{\varepsilon(\varrho_-)-\varepsilon(\varrho_1)}{\varrho_- -\varrho_1}\bigg) \\
	&\qquad\leq\delta_1^\star(\varrho_1)\,\varrho_1\,(v_{1\,2}^\star(\varrho_1)+v_{-\,2})-(\delta_1^\star(\varrho_1)+\delta_2)\,\frac{\varrho_-\,\varrho_1\,(v_{1\,2}^\star(\varrho_1)-v_{-\,2})}{\varrho_- -\varrho_1},
	\end{split} \label{eq:SRcond3} \\[0.5cm] 
	\begin{split}
	&(v_{+\,2}-v_{1\,2}^\star(\varrho_1))\,\bigg(p(\varrho_1)+p(\varrho_+)-2\,\varrho_1\,\varrho_+\,\frac{\varepsilon(\varrho_1)-\varepsilon(\varrho_+)}{\varrho_1 -\varrho_+}\bigg) \\
	&\qquad\leq-\delta_1^\star(\varrho_1)\,\varrho_1\,(v_{+\,2}+v_{1\,2}^\star(\varrho_1))+(\delta_1^\star(\varrho_1)+\delta_2)\,\frac{\varrho_1\,\varrho_+\,(v_{+\,2}-v_{1\,2}^\star(\varrho_1))}{\varrho_1 -\varrho_+},
	\end{split} \label{eq:SRcond4}
	\end{align}
	where we define the functions
	\begin{align}
	&v_{1\,2}^\star(\varrho_1) := \frac{1}{\varrho_1\,(\varrho_- - \varrho_+)}\,\Bigg(-\varrho_-\,v_{-\,2}\,(\varrho_+ - \varrho_1) - \varrho_+\,v_{+\,2}\,(\varrho_1 - \varrho_-) \notag\\
	&\ \ + \sqrt{\Big[(\varrho_- - \varrho_+)\,\big(p(\varrho_-) - p(\varrho_+)\big) - \varrho_+\,\varrho_-\,(v_{-\,2} - v_{+\,2})^2\Big]\,(\varrho_1 - \varrho_-)\,(\varrho_+ - \varrho_1)}\Bigg)  \label{eq:SRdefnv12}
	\end{align}
	and
	\begin{equation}
	\begin{split}
	&\delta_1^\star(\varrho_1) := -\frac{p(\varrho_1) - p(\varrho_-)}{\varrho_1} + \frac{\varrho_-\,(\varrho_1 - \varrho_-)}{\varrho_1^2\,(\varrho_- - \varrho_+)^2}\,\Bigg(\varrho_+\,(v_{-\,2} - v_{+\,2}) \\
	&\ \ + \sqrt{\Big[(\varrho_- - \varrho_+)\,\big(p(\varrho_-) - p(\varrho_+)\big) - \varrho_+\,\varrho_-\,(v_{-\,2} - v_{+\,2})^2\Big]\,\frac{\varrho_+ - \varrho_1}{\varrho_1 - \varrho_-}}\Bigg)^2.  
	\end{split}
	\label{eq:SRdefndelta1}
	\end{equation} 
	
	Note that these functions are well-defined for $\varrho_- < \varrho_1 < \varrho_+$ and for initial states $(\varrho_\pm,v_\pm)$ fulfilling \eqref{eq:1S3R}, which will be shown in the proof. 
	\label{thm:SRsubs}
\end{theorem}

\begin{remark}
	In this theorem we have an ``if and only if'' statement. This is the reason why we denote the functions defined in \eqref{eq:SRdefnv12} and \eqref{eq:SRdefndelta1} as $v_{1\,2}^\star,\delta_1^\star$ and not simply $v_{1\,2},\delta_1$. If an admissible fan subsolution is given, then it is a priori not clear that the $v_{1\,2},\delta_1$ given by the admissible fan subsolution are equal to the $v_{1\,2}^\star,\delta_1^\star$ defined in \eqref{eq:SRdefnv12} and \eqref{eq:SRdefndelta1}.
\end{remark} 

\begin{proof}
	Suppose there is an admissible fan subsolution. By proposition \ref{prop:2alggl} there exist constants $\mu_0,\mu_1\in\mathbb{R}$, $\rho_1\in\mathbb{R}^+$, $v_{1\,2}\in\mathbb{R}$ and $\delta_1,\delta_2\in\mathbb{R}$ such that \eqref{eq:2order}-\eqref{eq:2admr} hold. From \eqref{eq:2sc2} we have $\delta_2\in\mathbb{R}^+$. 
	
	Adding \eqref{eq:2rhl1} and \eqref{eq:2rhr1} and solving the result for $\mu_1$ leads to
	\begin{equation}
	\mu_1 = \frac{\rho_-\,v_{-\,2} - \rho_+\,v_{+\,2} - \mu_0\,(\rho_- - \rho_1)}{\rho_1 - \rho_+}.
	\label{eq:temp41}
	\end{equation}
	
	Next we add \eqref{eq:2rhl2} and \eqref{eq:2rhr2} and use \eqref{eq:2rhl1} and \eqref{eq:2rhr1} to obtain
	\begin{equation*}
	\mu_0^2\,(\rho_- - \rho_1) + \mu_1^2\,(\rho_1 - \rho_+) = \rho_-\,v_{-\,2}^2 - \rho_+\,v_{+\,2}^2 + p(\rho_-) - p(\rho_+).
	\end{equation*}
	If we use \eqref{eq:temp41} to eliminate $\mu_1$ and solve for $\mu_0$ we get
	\begin{equation}
	\begin{split}
	\mu_0 &= \frac{\rho_-\,v_{-\,2} - \rho_+\,v_{+\,2}}{\rho_- - \rho_+} \\
	&\quad\pm \frac{1}{\rho_- - \rho_+}\,\sqrt{\Big[(\rho_- - \rho_+)\,\big(p(\rho_-) - p(\rho_+)\big) - \rho_+\,\rho_-\,(v_{-\,2} - v_{+\,2})^2\Big]\,\frac{\rho_+ - \rho_1}{\rho_1 - \rho_-}}.
	\end{split}
	\label{eq:SRnu0pm}
	\end{equation}
	
	Using this result and \eqref{eq:temp41} one has
	\begin{equation}
	\begin{split}
	\mu_1 &= \frac{\rho_-\,v_{-\,2} - \rho_+\,v_{+\,2}}{\rho_- - \rho_+} \\
	&\quad\mp \frac{1}{\rho_- - \rho_+}\,\sqrt{\Big[(\rho_- - \rho_+)\,\big(p(\rho_-) - p(\rho_+)\big) - \rho_+\,\rho_-\,(v_{-\,2} - v_{+\,2})^2\Big]\,\frac{\rho_1 - \rho_-}{\rho_+ - \rho_1}},
	\end{split}
	\label{eq:SRnu1pm}
	\end{equation}
	where the signs in the last two equations have to be opposite. 
	
	Lemma \ref{lemma:introot} and equation \eqref{eq:1S3R} yield that 
	\begin{equation*}
	\frac{(\rho_- - \rho_+)\,\big(p(\rho_-)-p(\rho_+)\big)}{\rho_-\,\rho_+} > (v_{-\,2} - v_{+\,2})^2.
	\end{equation*}
	This is equivalent to 
	\begin{equation*}
	(\rho_- - \rho_+)\,\big(p(\rho_-) - p(\rho_+)\big) - \rho_+\,\rho_-\,(v_{-\,2} - v_{+\,2})^2>0.
	\end{equation*}
	
	Hence \eqref{eq:SRnu0pm} and \eqref{eq:SRnu1pm} yield that $\rho_+ - \rho_1$ and $\rho_1 - \rho_-$ have the same sign. Because $\rho_- < \rho_+$, we have $\rho_- < \rho_1 < \rho_+$, i.e. \eqref{eq:SRcond1}. Now we want to choose the correct signs in the equations for $\mu_0$ and $\mu_1$, i.e. in \eqref{eq:SRnu0pm} and \eqref{eq:SRnu1pm}. Assume we had a ``$-$'' in \eqref{eq:SRnu0pm} and therefore a ``$+$'' in \eqref{eq:SRnu1pm}. Then 
	\begin{equation*}
	\mu_0 > \frac{\rho_-\,v_{-\,2} - \rho_+\,v_{+\,2}}{\rho_- - \rho_+} > \mu_1,
	\end{equation*}
	since $\rho_- - \rho_+ <0$. This is a contradiction to \eqref{eq:2order}. Hence the proper sign in \eqref{eq:SRnu0pm} is ``$+$'' and in \eqref{eq:SRnu1pm} it is ``$-$'', i.e.
	\begin{align}
	\begin{split}
	\mu_0 &= \frac{\rho_-\,v_{-\,2} - \rho_+\,v_{+\,2}}{\rho_- - \rho_+} \\
	&\quad + \frac{1}{\rho_- - \rho_+}\,\sqrt{\Big[(\rho_- - \rho_+)\,\big(p(\rho_-) - p(\rho_+)\big) - \rho_+\,\rho_-\,(v_{-\,2} - v_{+\,2})^2\Big]\,\frac{\rho_+ - \rho_1}{\rho_1 - \rho_-}},
	\end{split}
	\label{eq:SRdefnnu0} \\
	\begin{split}
	\mu_1 &= \frac{\rho_-\,v_{-\,2} - \rho_+\,v_{+\,2}}{\rho_- - \rho_+} \\
	&\quad - \frac{1}{\rho_- - \rho_+}\,\sqrt{\Big[(\rho_- - \rho_+)\,\big(p(\rho_-) - p(\rho_+)\big) - \rho_+\,\rho_-\,(v_{-\,2} - v_{+\,2})^2\Big]\,\frac{\rho_1 - \rho_-}{\rho_+ - \rho_1}}.
	\end{split}
	\label{eq:SRdefnnu1}
	\end{align}
	
	Next we compute $v_{1\,2}$ using \eqref{eq:SRdefnnu0} and \eqref{eq:2rhl1} and get 
	\begin{align*}
	&v_{1\,2} = \frac{1}{\rho_1\,(\rho_- - \rho_+)}\,\Bigg(-\rho_-\,v_{-\,2}\,(\rho_+ - \rho_1) - \rho_+\,v_{+\,2}\,(\rho_1 - \rho_-) \\
	&\quad+ \sqrt{\Big[(\rho_- - \rho_+)\,\big(p(\rho_-) - p(\rho_+)\big) - \rho_+\,\rho_-\,(v_{-\,2} - v_{+\,2})^2\Big]\,(\rho_1 - \rho_-)\,(\rho_+ - \rho_1)}\Bigg).
	\end{align*}
	
	With \eqref{eq:2rhl2} we finally find 
	\begin{align*}
	\delta_1 &= -\frac{p(\rho_1) - p(\rho_-)}{\rho_1} + \frac{\rho_-\,(\rho_1 - \rho_-)}{\rho_1^2\,(\rho_- - \rho_+)^2}\,\Bigg(\rho_+\,(v_{-\,2} - v_{+\,2}) \\
	&\quad+ \sqrt{\Big[(\rho_- - \rho_+)\,\big(p(\rho_-) - p(\rho_+)\big) - \rho_+\,\rho_-\,(v_{-\,2} - v_{+\,2})^2\Big]\,\frac{\rho_+ - \rho_1}{\rho_1 - \rho_-}}\Bigg)^2.
	\end{align*}
	
	Hence we have $\delta_1=\delta_1^\star(\varrho_1)$ and $v_{1\,2}=v_{1\,2}^\star(\varrho_1)$. From \eqref{eq:2sc1} we obtain \eqref{eq:SRcond2} and the admissi\-bi\-li\-ty conditions \eqref{eq:2adml} and \eqref{eq:2admr} yield \eqref{eq:SRcond3} and \eqref{eq:SRcond4}. 
	
	It remains to prove the converse. Let $\rho_1,\delta_2\in\mathbb{R}^+$ such that \eqref{eq:SRcond1} - \eqref{eq:SRcond4} hold. Define $v_{1\,2}=v_{1\,2}^\star(\varrho_1)$, $\delta_1=\delta_1^\star(\varrho_1)$ and $\mu_0,\mu_1$ through \eqref{eq:SRdefnnu0}, resp. \eqref{eq:SRdefnnu1}. By easy computations one can check that $\rho_1,\delta_2$ together with $\mu_0,\mu_1,v_{1\,2},\delta_1$ fulfill the conditions \eqref{eq:2order} - \eqref{eq:2admr} and therefore define an admissible fan subsolution according to proposition \ref{prop:2alggl}. \qed
\end{proof}

As already mentioned it turns out that there does not always exist an admissible fan subsolution. Nevertheless we can prove existence of infinitely many solutions. The idea is to work with an auxiliary state.

\subsection{An auxiliary state}

\begin{theorem}
	Assume that \eqref{eq:1S3R} holds. Then there exist infinitely many admissible weak solutions to \eqref{eq:euler}, \eqref{eq:ourinit}.
	\label{thm:SRexistence}
\end{theorem}

For convenience we will from now on use the notation $\mathcal{P}:=\mathbb{R}^+\times\mathbb{R}^2$ for the phase space and $U:=(\rho,v)\in\mathcal{P}$ for a state. 

\begin{definition}
	Consider the 3-dimensional phase space $\mathcal{P}=\mathbb{R}^+\times\mathbb{R}^2$. We denote a 2-dimensional ball with center $\widetilde{U}_M=(\widetilde{\varrho}_M,\widetilde{v}_M)\in\mathcal{P}$ and radius $r>0$ as 
	\begin{equation*}
	B_r(\widetilde{U}_M):=\big\{(\varrho,v)\in\mathcal{P}\,\big|\,v_1=\widetilde{v}_{M\,1},\|(\varrho,v)-(\widetilde{\varrho}_M,\widetilde{v}_M)\|<r\big\}.
	\end{equation*}
\end{definition}

To prove the theorem we will need the following lemma. We will forget about the given initial states $U_-=(\varrho_-,v_-)$, $U_+=(\varrho_+,v_+)$ for a moment.

\begin{lemma}
	Let $\widetilde{U}_-=(\widetilde{\varrho}_-,\widetilde{v}_-)\in\mathcal{P}$ be any given state and $\widetilde{U}_M=(\widetilde{\varrho}_M,\widetilde{v}_M)\in\mathcal{P}$ a state that can be connected to $\widetilde{U}_-$ by a 1-shock. Then there exists a radius $r>0$ with the following property: \\
	If $\widetilde{U}_+=(\widetilde{\varrho}_+,\widetilde{v}_+)\in\mathcal{P}$ is a state that fulfills
	\begin{itemize}
		\item $\widetilde{\varrho}_+>\widetilde{\varrho}_M$,
		\item $\widetilde{U}_+\in B_r(\widetilde{U}_M)$ and
		\item the standard solution to the problem \eqref{eq:euler}, \eqref{eq:ourinit} with $\widetilde{U}_-$ and $\widetilde{U}_+$ as initial states consists of a 1-shock and a 3-rarefaction,
	\end{itemize}
	then there exists an admissible fan subsolution to the problem \ref{eq:euler}, \ref{eq:ourinit} with $\widetilde{U}_-$ and $\widetilde{U}_+$ as initial states. In addition to that the density $\varrho_1$ that appears in the admissible fan subsolution fulfills $\varrho_1<\widetilde{\varrho}_M$.
	\label{lemma:SRperturb}
\end{lemma}

During the workshop ``Ideal Fluids and Transport'' at IMPAN in Warsaw (February 13-15, 2017) the authors learned about a result achieved by E. Chiodaroli and O. Kreml which is similar to our lemma \ref{lemma:SRperturb}, see also \cite{chio17}.

\begin{proof}
	To prove this we are going to use theorem \ref{thm:SRsubs}. Hence it suffices to show that there exists a radius $r>0$ such that for every state  $\widetilde{U}_+\in B_r(\widetilde{U}_M)$ with $\widetilde{\varrho}_+>\widetilde{\varrho}_M$, we find $\varrho_1,\delta_2\in\mathbb{R}^+$ such that inequalities \eqref{eq:SRcond1} - \eqref{eq:SRcond4} are fulfilled. 
	
	In view of the functions $v_{1\,2}^\star$ and $\delta_1^\star$ (see \eqref{eq:SRdefnv12}, \eqref{eq:SRdefndelta1}), we define the following functions $\delta_1^\diamond,v_{1\,2}^\diamond:\mathbb{R}^+\times\mathcal{P}\rightarrow\mathbb{R}$ as
	\begin{equation*}
	\begin{split}
	&v_{1\,2}^\diamond(\varrho_1,\widetilde{U}_+) := 				\frac{1}{\varrho_1\,(\widetilde{\varrho}_- - \widetilde{\varrho}_+)}\,\Bigg(-\widetilde{\varrho}_-\,\widetilde{v}_{-\,2}\,(\widetilde{\varrho}_+ - \varrho_1) - \widetilde{\varrho}_+\,\widetilde{v}_{+\,2}\,(\varrho_1 - \widetilde{\varrho}_-) \\
	&\ \ + \sqrt{\Big[(\widetilde{\varrho}_- - \widetilde{\varrho}_+)\,\big(p(\widetilde{\varrho}_-) - p(\widetilde{\varrho}_+)\big) - \widetilde{\varrho}_+\,\widetilde{\varrho}_-\,(\widetilde{v}_{-\,2} - \widetilde{v}_{+\,2})^2\Big]\,(\varrho_1 - \widetilde{\varrho}_-)\,(\widetilde{\varrho}_+ - \varrho_1)}\Bigg)  
	\end{split}
	\end{equation*}
	and
	\begin{equation*}
	\begin{split}
	&\delta_1^\diamond(\varrho_1,\widetilde{U}_+) := -\frac{p(\varrho_1) - p(\widetilde{\varrho}_-)}{\varrho_1} + \frac{\widetilde{\varrho}_-\,(\varrho_1 - \widetilde{\varrho}_-)}{\varrho_1^2\,(\widetilde{\varrho}_- - \widetilde{\varrho}_+)^2}\,\Bigg(\widetilde{\varrho}_+\,(\widetilde{v}_{-\,2} - \widetilde{v}_{+\,2}) \\
	&\qquad+ \sqrt{\Big[(\widetilde{\varrho}_- - \widetilde{\varrho}_+)\,\big(p(\widetilde{\varrho}_-) - p(\widetilde{\varrho}_+)\big) - \widetilde{\varrho}_+\,\widetilde{\varrho}_-\,(\widetilde{v}_{-\,2} - \widetilde{v}_{+\,2})^2\Big]\,\frac{\widetilde{\varrho}_+ - \varrho_1}{\varrho_1 - \widetilde{\varrho}_-}}\Bigg)^2.  
	\end{split}
	\end{equation*}
	
	In addition we define functions $A,B:\mathbb{R}^+\times\mathbb{R}^+\times\mathcal{P}\rightarrow\mathbb{R}$ as
	\begin{align*}
	A(\varrho_1,\delta_2,\widetilde{U}_+) &:= \delta_1^\diamond(\varrho_1,\widetilde{U}_+)\,\varrho_1\,\big(v_{1\,2}^\diamond(\varrho_1,\widetilde{U}_+)+\widetilde{v}_{-\,2}\big) \\
	&\ \ \ -\big(\delta_1^\diamond(\varrho_1,\widetilde{U}_+)+\delta_2\big)\,\frac{\widetilde{\varrho}_-\,\varrho_1\,\big(v_{1\,2}^\diamond(\varrho_1,\widetilde{U}_+)-\widetilde{v}_{-\,2}\big)}{\widetilde{\varrho}_- -\varrho_1} \\
	&\ \ \  - \big(v_{1\,2}^\diamond(\varrho_1,\widetilde{U}_+)-\widetilde{v}_{-\,2}\big)\,\bigg(p(\widetilde{\varrho}_-)+p(\varrho_1)-2\,\widetilde{\varrho}_-\,\varrho_1\,\frac{\varepsilon(\widetilde{\varrho}_-)-\varepsilon(\varrho_1)}{\widetilde{\varrho}_- -\varrho_1}\bigg), \\
	B(\varrho_1,\delta_2,\widetilde{U}_+) &:= -\delta_1^\diamond(\varrho_1,\widetilde{U}_+)\,\varrho_1\,\big(\widetilde{v}_{+\,2}+v_{1\,2}^\diamond(\varrho_1,\widetilde{U}_+)\big) \\
	&\ \ \ +\big(\delta_1^\diamond(\varrho_1,\widetilde{U}_+)+\delta_2\big)\,\frac{\varrho_1\,\widetilde{\varrho}_+\,\big(\widetilde{v}_{+\,2}-v_{1\,2}^\diamond(\varrho_1,\widetilde{U}_+)\big)}{\varrho_1 -\widetilde{\varrho}_+} \\
	&\ \ \  - \big(\widetilde{v}_{+\,2}-v_{1\,2}^\diamond(\varrho_1,\widetilde{U}_+)\big)\,\bigg(p(\varrho_1)+p(\widetilde{\varrho}_+)-2\,\varrho_1\,\widetilde{\varrho}_+\,\frac{\varepsilon(\varrho_1)-\varepsilon(\widetilde{\varrho}_+)}{\varrho_1 -\widetilde{\varrho}_+}\bigg).
	\end{align*}
	
	Since $\widetilde{U}_-$ and $\widetilde{U}_M$ can be connected by a 1-shock we obtain according to proposition \ref{prop:standardsolution} that $\widetilde{\varrho}_-<\widetilde{\varrho}_M$ and 
	\begin{equation}
	\widetilde{v}_{-\,2} - \widetilde{v}_{M\,2} = \sqrt{\frac{(\widetilde{\varrho}_M - \widetilde{\varrho}_-)\,\big(p(\widetilde{\varrho}_M) - p(\widetilde{\varrho}_-)\big)}{\widetilde{\varrho}_M\,\widetilde{\varrho}_-}}.
	\label{eq:temp42}
	\end{equation} 
	
	Next we show that there exists $\varrho_1\in(\widetilde{\varrho}_-,\widetilde{\varrho}_M)$ such that 
	\begin{align}
	&\delta_1^\diamond(\varrho_1,\widetilde{U}_+=\widetilde{U}_M) >0, \label{eq:SRconddeltaABd}\\
	&A(\varrho_1,\delta_2=0,\widetilde{U}_+=\widetilde{U}_M) >0, \label{eq:SRconddeltaABA}\\
	&B(\varrho_1,\delta_2=0,\widetilde{U}_+=\widetilde{U}_M) >0. \label{eq:SRconddeltaABB}
	\end{align}
	
	First we prove that \eqref{eq:SRconddeltaABd} is true for all $\varrho_1\in(\widetilde{\varrho}_-,\widetilde{\varrho}_M)$. Using \eqref{eq:temp42} we obtain
	\begin{equation*}
	\delta_1^\diamond(\varrho_1,\widetilde{U}_+=\widetilde{U}_M) = -\frac{p(\varrho_1) - p(\widetilde{\varrho}_-)}{\varrho_1} + \frac{\widetilde{\varrho}_M}{\varrho_1^2}\ \frac{p(\widetilde{\varrho}_M) - p(\widetilde{\varrho}_-)}{\widetilde{\varrho}_M - \widetilde{\varrho}_-}\ (\varrho_1 - \widetilde{\varrho}_-).
	\end{equation*}
	Each $\varrho_1\in(\widetilde{\varrho}_-,\widetilde{\varrho}_M)$ can be written as a convex combination of $\widetilde{\varrho}_-$ and $\widetilde{\varrho}_M$. In other words there exists $\theta\in(0,1)$ such that 
	\begin{equation*}
	\varrho_1=\theta\,\widetilde{\varrho}_- + (1-\theta)\,\widetilde{\varrho}_M.
	\end{equation*}
	Since $p$ is a convex function of $\varrho$ we have 
	\begin{equation*}
	p(\varrho_1)=p\big(\theta\,\widetilde{\varrho}_- + (1-\theta)\,\widetilde{\varrho}_M\big) \leq \theta\,p(\widetilde{\varrho}_-) + (1-\theta)\,p(\widetilde{\varrho}_M)
	\end{equation*}
	and hence
	\begin{align*}
	\delta_1^\diamond(\varrho_1,\widetilde{U}_+=\widetilde{U}_M) &= \frac{1}{\varrho_1}\,\bigg( -p(\varrho_1) + p(\widetilde{\varrho}_-) + \frac{\widetilde{\varrho}_M}{\varrho_1}\ \frac{p(\widetilde{\varrho}_M) - p(\widetilde{\varrho}_-)}{\widetilde{\varrho}_M - \widetilde{\varrho}_-}\ (\varrho_1 - \widetilde{\varrho}_-)\bigg) \\
	&\geq \frac{1}{\varrho_1^2}\,\theta\,(1-\theta)\,\big(p(\widetilde{\varrho}_M) - p(\widetilde{\varrho}_-)\big)\,(\widetilde{\varrho}_M - \widetilde{\varrho}_-)\ >\ 0.
	\end{align*}
	Therefore \eqref{eq:SRconddeltaABd} is true for all $\varrho_1\in(\widetilde{\varrho}_-,\widetilde{\varrho}_M)$. 
	
	For convenience we define
	\begin{equation*}
	R := \sqrt{\frac{(\widetilde{\varrho}_M - \widetilde{\varrho}_-)\,\big(p(\widetilde{\varrho}_M) - p(\widetilde{\varrho}_-)\big)}{\widetilde{\varrho}_M\,\widetilde{\varrho}_-}}.
	\end{equation*}
	
	To show the existence of $\varrho_1\in(\widetilde{\varrho}_-,\widetilde{\varrho}_M)$ that satisfies \eqref{eq:SRconddeltaABA} and \eqref{eq:SRconddeltaABB} we consider two cases: Let first \begin{equation*}
	\widetilde{v}_{-\,2}>\frac{\widetilde{\varrho}_M}{2\,(\widetilde{\varrho}_M-\widetilde{\varrho}_-)}\,R .
	\end{equation*}
	An easy computation leads to
	\begin{equation*}
	\lim\limits_{\varrho_1\rightarrow\widetilde{\varrho}_-} A(\varrho_1,\delta_2=0,\widetilde{U}_+=\widetilde{U}_M)=0, 
	\end{equation*}
	and also
	\begin{align*} 
	\lim\limits_{\varrho_1\rightarrow\widetilde{\varrho}_-}&\bigg(\frac{\partial}{\partial\varrho_1}A(\varrho_1,\delta_2=0,\widetilde{U}_+=\widetilde{U}_M)\bigg)\\
	&=\bigg(-\frac{\widetilde{\varrho}_M}{\widetilde{\varrho}_M -\widetilde{\varrho}_-}\,R + 2\,\widetilde{v}_{-\,2}\bigg)\,\bigg(-p'(\widetilde{\varrho}_-) + \frac{\widetilde{\varrho}_M}{\widetilde{\varrho}_-}\ \ \frac{p(\widetilde{\varrho}_M) - p(\widetilde{\varrho}_-)}{\widetilde{\varrho}_M - \widetilde{\varrho}_-}\bigg).
	\end{align*}
	
	In the case under consideration it holds that 
	\begin{equation*}
	-\frac{\widetilde{\varrho}_M}{\widetilde{\varrho}_M -\widetilde{\varrho}_-}\,R + 2\,v_{-\,2} > -\frac{\widetilde{\varrho}_M}{\widetilde{\varrho}_M -\widetilde{\varrho}_-}\,R + \frac{\widetilde{\varrho}_M}{\widetilde{\varrho}_M-\widetilde{\varrho}_-}\,R = 0.
	\end{equation*}
	In addition to that the fact that $\widetilde{\varrho}_-<\widetilde{\varrho}_M$ and the convexity of $p$ lead to 
	\begin{equation*}
	-p'(\widetilde{\varrho}_-) + \frac{\widetilde{\varrho}_M}{\widetilde{\varrho}_-}\ \ \frac{p(\widetilde{\varrho}_M) - p(\widetilde{\varrho}_-)}{\widetilde{\varrho}_M - \widetilde{\varrho}_-} > -p'(\widetilde{\varrho}_-) + \frac{p(\widetilde{\varrho}_M) - p(\widetilde{\varrho}_-)}{\widetilde{\varrho}_M - \widetilde{\varrho}_-} \geq 0.
	\end{equation*}
	Hence
	\begin{equation*} 
	\lim\limits_{\varrho_1\rightarrow\widetilde{\varrho}_-}\bigg(\frac{\partial}{\partial\varrho_1}A(\varrho_1,\delta_2=0,\widetilde{U}_+=\widetilde{U}_M)\bigg) > 0.
	\end{equation*}
	
	By obvious continuity of the function $A$ there exists $\varrho_1\in(\widetilde{\varrho}_-,\widetilde{\varrho}_M)$ where $\varrho_1\approx\widetilde{\varrho}_-$ such that \eqref{eq:SRconddeltaABA} holds. 
	
	Another computation shows that 
	\begin{align*}
	&B(\varrho_1=\widetilde{\varrho}_-,\delta_2=0,\widetilde{U}_+=\widetilde{U}_M) \\
	&=R\,\bigg(p(\varrho_-)+p(\varrho_M) - 2\,\varrho_M\,\varrho_- \,\frac{\varepsilon(\varrho_M)-\varepsilon(\varrho_-)}{\varrho_M - \varrho_-}\bigg)>0,
	\end{align*} 
	according to lemma \ref{lemma:auschiokreml}. Hence by continuity of $B$ we can choose $\varrho_1\in(\widetilde{\varrho}_-,\widetilde{\varrho}_M)$ such that \eqref{eq:SRconddeltaABB} is fulfilled in addition to \eqref{eq:SRconddeltaABA}. 
	
	Suppose now the second case
	\begin{equation*}
	\widetilde{v}_{-\,2} \leq\frac{\widetilde{\varrho}_M}{2\,(\widetilde{\varrho}_M-\widetilde{\varrho}_-)}\,R .
	\end{equation*}
	Similar computations yield
	\begin{align*}
	&A(\varrho_1=\widetilde{\varrho}_M,\delta_2=0,\widetilde{U}_+=\widetilde{U}_M) \\
	&=R\,\bigg(p(\varrho_-)+p(\varrho_M) - 2\,\varrho_M\,\varrho_- \,\frac{\varepsilon(\varrho_M)-\varepsilon(\varrho_-)}{\varrho_M - \varrho_-}\bigg)>0,
	\end{align*}
	and furthermore
	\begin{equation*} 
	\lim\limits_{\varrho_1\rightarrow\widetilde{\varrho}_M} B(\varrho_1,\delta_2=0,\widetilde{U}_+=\widetilde{U}_M)=0, 
	\end{equation*}
	together with
	\begin{align*}
	\lim\limits_{\varrho_1\rightarrow\widetilde{\varrho}_M}&\bigg(\frac{\partial}{\partial\varrho_1}B(\varrho_1,\delta_2=0,\widetilde{U}_+=\widetilde{U}_M)\bigg)\\
	&=\bigg(-\frac{2\,\widetilde{\varrho}_M-\widetilde{\varrho}_-}{\widetilde{\varrho}_M -\widetilde{\varrho}_-}\,R + 2\,\widetilde{v}_{-\,2}\bigg)\,\bigg(p'(\widetilde{\varrho}_M) - \frac{\widetilde{\varrho}_-}{\widetilde{\varrho}_M}\ \ \frac{p(\widetilde{\varrho}_M) - p(\widetilde{\varrho}_-)}{\widetilde{\varrho}_M - \widetilde{\varrho}_-}\bigg).
	\end{align*}
	
	In the considered case we have
	\begin{equation*}
	-\frac{2\,\widetilde{\varrho}_M-\widetilde{\varrho}_-}{\widetilde{\varrho}_M -\widetilde{\varrho}_-}\,R + 2\,\widetilde{v}_{-\,2} \leq -\frac{2\,\widetilde{\varrho}_M-\widetilde{\varrho}_-}{\widetilde{\varrho}_M -\widetilde{\varrho}_-}\,R + \frac{\widetilde{\varrho}_M}{\widetilde{\varrho}_M-\widetilde{\varrho}_-}\,R = -R < 0.
	\end{equation*}
	Additionally the convexity of $p$ and $\widetilde{\varrho}_-<\widetilde{\varrho}_M$ lead to
	\begin{equation*}
	p'(\widetilde{\varrho}_M) - \frac{\widetilde{\varrho}_-}{\widetilde{\varrho}_M}\ \ \frac{p(\widetilde{\varrho}_M) - p(\widetilde{\varrho}_-)}{\widetilde{\varrho}_M - \widetilde{\varrho}_-} > p'(\widetilde{\varrho}_M) - \frac{p(\widetilde{\varrho}_M) - p(\widetilde{\varrho}_-)}{\widetilde{\varrho}_M - \widetilde{\varrho}_-} \geq 0.
	\end{equation*}
	Hence
	\begin{equation*} 
	\lim\limits_{\varrho_1\rightarrow\widetilde{\varrho}_M}\bigg(\frac{\partial}{\partial\varrho_1}B(\varrho_1,\delta_2=0,\widetilde{U}_+=\widetilde{U}_M)\bigg) < 0
	\end{equation*}
	and therefore by continuity of $A$ and $B$ there exists $\varrho_1\in(\widetilde{\varrho}_-,\widetilde{\varrho}_M)$ such that \eqref{eq:SRconddeltaABA} and \eqref{eq:SRconddeltaABB} hold, where $\varrho_1\approx\widetilde{\varrho}_M$. 
	
	By continuity we can find $\delta_2>0$ in addition to $\varrho_1$ found above, such that 
	\begin{align*}
	&\delta_1^\diamond(\varrho_1,\widetilde{U}_+=\widetilde{U}_M) >0, \\
	&A(\varrho_1,\delta_2,\widetilde{U}_+=\widetilde{U}_M) >0, \\
	&B(\varrho_1,\delta_2,\widetilde{U}_+=\widetilde{U}_M) >0. 
	\end{align*}
	
	Again by continuity there exists a radius $r>0$ such that 
	\begin{align}
	&\delta_1^\diamond(\varrho_1,\widetilde{U}_+)>0, \label{eq:dg0} \\
	&A(\varrho_1,\delta_2,\widetilde{U}_+) >0, \label{eq:Ag0}\\
	&B(\varrho_1,\delta_2,\widetilde{U}_+) >0 \label{eq:Bg0}
	\end{align}
	hold for all $\widetilde{U}_+\in B_r(\widetilde{U}_M)$. In other words for all $\widetilde{U}_+\in B_r(\widetilde{U}_M)$ we can find $\varrho_1,\delta_2\in\mathbb{R}^+$ such that $\widetilde{\varrho}_-<\varrho_1<\widetilde{\varrho}_M$ and \eqref{eq:dg0} - \eqref{eq:Bg0} are true. By assumption we have $\widetilde{\varrho}_M<\widetilde{\varrho}_+$ and hence \eqref{eq:SRcond1} is true. Additionally \eqref{eq:SRcond2} holds because of \eqref{eq:dg0} and finally \eqref{eq:Ag0}, resp. \eqref{eq:Bg0} imply \eqref{eq:SRcond3}, resp. \eqref{eq:SRcond4}. \qed
\end{proof} 

Next we prove theorem \ref{thm:SRexistence}.
\begin{proof}
	Let $U_M$ be the intermediate state of the standard solution. In other words $U_M$ lies on the 1-shock curve of the state $U_-$. So we can apply lemma \ref{lemma:SRperturb} to obtain a radius $r>0$. We fix a state $U_2\in\mathcal{P}$ such that
	\begin{itemize}
		\item $\varrho_M<\varrho_2<\varrho_+$,
		\item $$v_{2\,2}=v_{M\,2} + \int_{\varrho_M}^{\varrho_2} \frac{\sqrt{p'(r)}}{r}\,\dd r$$ and 
		\item $U_2 \in B_r(U_M)$.
	\end{itemize}
	
	Note that such a state $U_2$ exists. 
	
	Then consider the two new problems 
	\begin{align*}
	\widetilde{U}_- &= U_- \\
	\widetilde{U}_+ &= U_2,
	\end{align*}
	called problem $\sim$, and 
	\begin{align*}
	\widehat{U}_- &= U_2 \\
	\widehat{U}_+ &= U_+,
	\end{align*}
	which we call problem $ \wedge$. 
	
	Let us first consider problem $\sim$. It is easy to check that the standard solution of problem $\sim$ consists of a 1-shock and a 3-rarefaction using propostion \ref{prop:standardsolution}: We have $\varrho_-<\varrho_M$ and $\varrho_M<\varrho_2$ and hence $\varrho_-<\varrho_2$. In addition to that it holds that
	\begin{align*}
	v_{2\,2} - v_{-\,2} &= v_{M\,2} - v_{-\,2} + \int_{\varrho_M}^{\varrho_2} \frac{\sqrt{p'(r)}}{r}\,\dd r \\
	&= -\sqrt{\frac{\big(\varrho_M - \varrho_-\big)\,\big(p(\varrho_M) - p(\varrho_-)\big)}{\varrho_M\,\varrho_-}} + \int_{\varrho_M}^{\varrho_2} \frac{\sqrt{p'(r)}}{r}\,\dd r \\
	&< \int_{\varrho_M}^{\varrho_2} \frac{\sqrt{p'(r)}}{r}\,\dd r\quad<\quad \int_{\varrho_-}^{\varrho_2} \frac{\sqrt{p'(r)}}{r}\,\dd r
	\end{align*}
	and 
	\begin{align*}
	v_{2\,2} &- v_{-\,2} = -\sqrt{\frac{\big(\varrho_M - \varrho_-\big)\,\big(p(\varrho_M) - p(\varrho_-)\big)}{\varrho_M\,\varrho_-}} + \int_{\varrho_M}^{\varrho_2} \frac{\sqrt{p'(r)}}{r}\,\dd r \\
	&\quad> -\sqrt{\frac{\big(\varrho_M - \varrho_-\big)\,\big(p(\varrho_M) - p(\varrho_-)\big)}{\varrho_M\,\varrho_-}} \ >\  -\sqrt{\frac{\big(\varrho_2 - \varrho_-\big)\,\big(p(\varrho_2) - p(\varrho_-)\big)}{\varrho_2\,\varrho_-}},
	\end{align*}
	where the last inequality comes from lemma \ref{lemma:rootroot}. 
	
	Hence we showed that the standard solution to problem $\sim$ consists of a 1-shock and a 3-rarefaction wave. 
	
	Because $U_2\in B_r(U_M)$ and $\varrho_2>\varrho_M$, according to lemma \ref{lemma:SRperturb} there exists an admissible fan subsolution to problem $\sim$ and hence infinitely many admissible weak solutions. In addition to that the same lemma yields $\varrho_1<\varrho_M$. 
	
	Now consider problem $\wedge$. We are going to prove that the standard solution to problem $\wedge$ consists only of a 3-rarefaction using proposition \ref{prop:standardsolution}. By definition of $U_2$ we have $\varrho_2<\varrho_+$ and additionally
	\begin{equation*}
	v_{+\,2} - v_{2\,2} = v_{+\,2} - v_{M\,2} - \int_{\varrho_M}^{\varrho_2} \frac{\sqrt{p'(r)}}{r}\,\dd r = \int_{\varrho_2}^{\varrho_+} \frac{\sqrt{p'(r)}}{r}\,\dd r.
	\end{equation*}	
	This shows that the standard solution of problem $\wedge$ consists of a just a 3-rarefaction wave. 
	
	To conclude we put together the wild solutions to problem $\sim$ and the standard solution to problem $\wedge$. To do this it remains to show that $\mu_1<\mu_2$ where $\mu_1$ is the speed of the right interface of the wild solutions of problem $\sim$ and $\mu_2=\lambda_3(U_2)$ is the left border of the rarefaction wave of the standard solution to problem $\wedge$. Here $\lambda_3(U)=v_2 + \sqrt{p'(\varrho)}$ denotes the 3rd eigenvalue of the Euler system, see \cite[equation (2.3)]{chio14}. 
	
	Since we have an admissible fan subsolution, we can apply proposition \ref{prop:2alggl}. Hence we get from \eqref{eq:2rhr2}
	\begin{equation} 
	\delta_1=\frac{\mu_1}{\varrho_1}\,(\varrho_1\,v_{1\,2} - \varrho_2\,v_{2\,2}) + \frac{\varrho_2}{\varrho_1}\,v_{2\,2}^2 - \frac{p(\varrho_1) - p(\varrho_2)}{\varrho_1} - v_{1\,2}^2
	\label{eq:temp43}
	\end{equation}
	and from \eqref{eq:2rhr1}
	\begin{equation*}
	v_{1\,2}=\frac{1}{\varrho_1}\big(\mu_1\,(\varrho_1-\varrho_2)+\varrho_2\,v_{2\,2}\big).
	\end{equation*}
	We use the latter to eliminate $v_{1\,2}$ in \eqref{eq:temp43} and obtain after some calculation
	\begin{equation*}
	\delta_1=\frac{\varrho_1 - \varrho_2}{\varrho_1^2}\,\varrho_2\,(\mu_1 - v_{2\,2})^2 - \frac{p(\varrho_1) - p(\varrho_2)}{\varrho_1}.
	\end{equation*}
	Since $\delta_1>0$, see \eqref{eq:2sc1}, it follows that 
	\begin{equation*}
	\frac{\varrho_1 - \varrho_2}{\varrho_1^2}\,\varrho_2\,(\mu_1 - v_{2\,2})^2 - \frac{p(\varrho_1) - p(\varrho_2)}{\varrho_1} > 0.
	\end{equation*}
	Because $\varrho_1<\varrho_M$ and $\varrho_M<\varrho_2$, we have $\varrho_1 - \varrho_2 < 0$. Therefore the inequality above is equivalent to 
	\begin{equation*}
	(\mu_1 - v_{2\,2})^2 < \frac{\varrho_1}{\varrho_2}\ \ \frac{p(\varrho_1) - p(\varrho_2)}{\varrho_1 - \varrho_2}.
	\end{equation*}
	
	Hence
	\begin{equation*}
	\mu_1 < v_{2\,2} + \sqrt{\frac{\varrho_1}{\varrho_2}\ \ \frac{p(\varrho_1) - p(\varrho_2)}{\varrho_1 - \varrho_2}} < v_{2\,2} + \sqrt{\frac{p(\varrho_1) - p(\varrho_2)}{\varrho_1 - \varrho_2}} \leq v_{2\,2} + \sqrt{p'(\rho_2)}
	\end{equation*}
	where the last inequality follows from the convexity of $p$. Since
	\begin{equation*}
	\mu_2 = \lambda_3(U_2) = v_{2\,2} + \sqrt{p'(\rho_2)}
	\end{equation*}
	we found the desired inequality $\mu_1 < \mu_2$. \qed
\end{proof}

The proof of theorem \ref{thm:SRexistence} yields admissible weak solutions of the form illustrated in figure \ref{fig:SR}.

\begin{figure}[hbt] 
	\centering
	\includegraphics[width=0.85\textwidth]{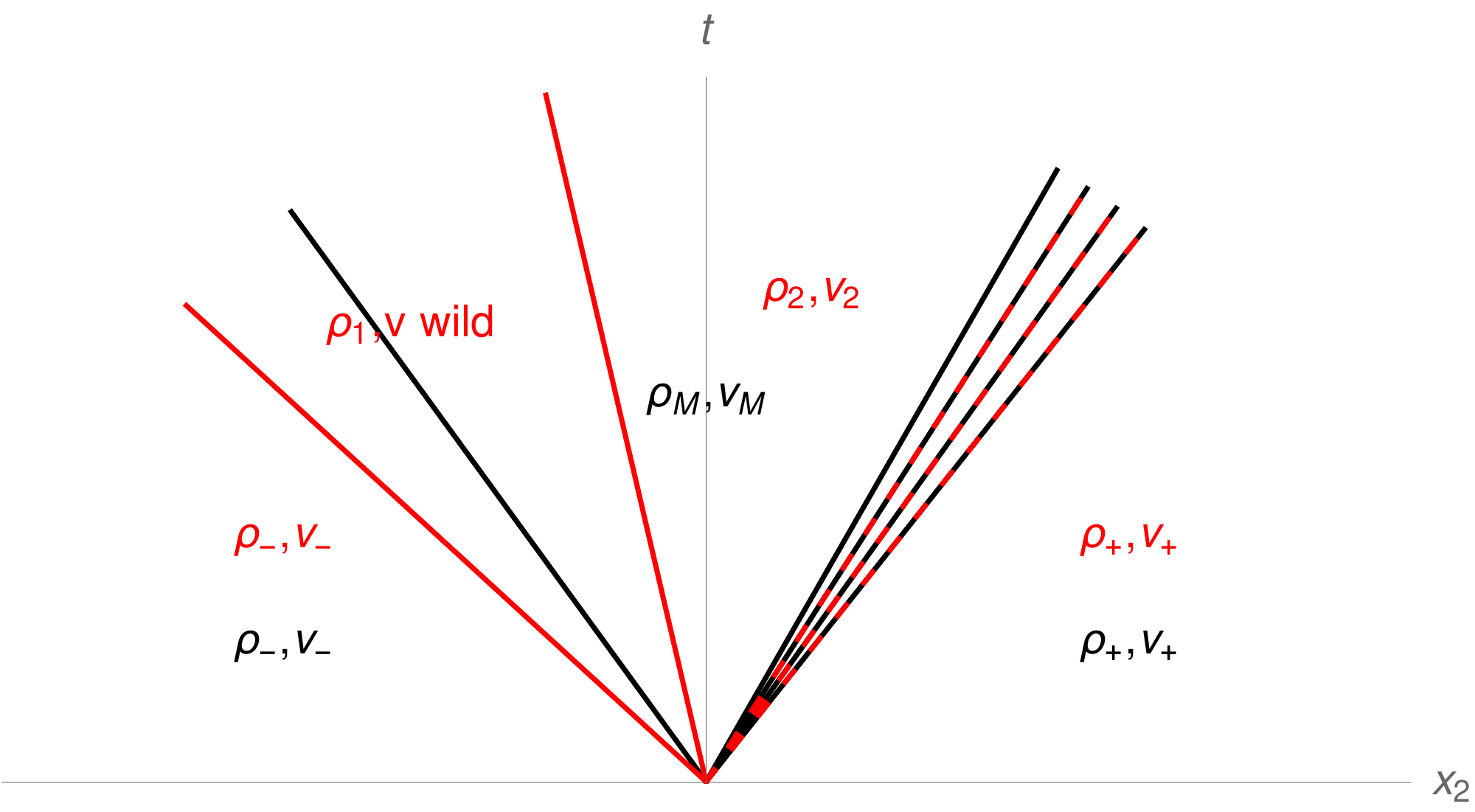}
	\caption{Structure of the standard solution (black) and of the admissible weak solutions produced in the proof of theorem \ref{thm:SRexistence} (red) } 
	\label{fig:SR}
\end{figure}


\section{The standard solution consists of just one shock}
\label{section:S}

What remains is the proof of existence of infinitely many admissible weak solutions in the case where the standard solution consists of just one shock.

\begin{remark} 
	As in the case of one shock and one rarefaction, it is enough to consider the case where the standard solution consists of a 1-shock because of the rotational invariance of the Euler system. If we have to deal with a 3-shock, we just rotate the coordinate system 180 degrees to obtain a new initial data
	\begin{equation}
	\begin{split}
	U_{-\,\text{new}} &= (\varrho_+,-v_+) \\
	U_{+\,\text{new}} &= (\varrho_-,-v_-).
	\end{split}
	\label{eq:rotinitS}
	\end{equation}
	Again, note that the sign of the velocities changes during this transformation. Proposition \ref{prop:standardsolution} yields then that the standard solution to the problem with rotated initial data \eqref{eq:rotinitS} consists of a 1-shock.
\end{remark} 

Let the initial values $\varrho_\pm\in\mathbb{R}^+$ and $v_\pm\in\mathbb{R}^2$ be such that the standard solution consists of a 1-shock. By proposition \ref{prop:standardsolution} this means, that 
\begin{equation}
\begin{split}
\varrho_-&<\varrho_+\quad\text{ and}\\
v_{+\,2} - v_{-\,2} &= -\sqrt{\frac{(\varrho_- - \varrho_+)\,\big(p(\varrho_-)-p(\varrho_+)\big)}{\varrho_-\,\varrho_+}}.
\end{split}
\label{eq:1S}
\end{equation}

\begin{theorem}
	Assume that \eqref{eq:1S} holds. Then there exist infinitely many admissible weak solutions to \eqref{eq:euler}, \eqref{eq:ourinit}.
	\label{thm:Sexistence}
\end{theorem}

\begin{proof}
	We apply lemma \ref{lemma:SRperturb} to $\widetilde{U}_-=U_-$ and $\widetilde{U}_M=U_+$ to obtain a radius $r>0$. We fix a state $U_2\in\mathcal{P}$ such that
	\begin{itemize}
		\item $\varrho_+<\varrho_2$,
		\item $v_{2\,2} = v_{+\,2} + \sqrt{\frac{(\varrho_2 - \varrho_+)\,\big(p(\varrho_2)-p(\varrho_+)\big)}{\varrho_2\,\varrho_+}}$,
		\item $U_2 \in B_r(U_+)$ and
		\item 
		\begin{equation}
		\sqrt{\frac{\big(\varrho_2 - \varrho_+\big)\,\big(p(\varrho_2)-p(\varrho_+)\big)}{\varrho_2\,\varrho_+}} < \int_{\varrho_-}^{\varrho_+} \frac{\sqrt{p'(r)}}{r}\,\dd r.
		\label{eq:temp51}
		\end{equation}
	\end{itemize}
	
	Note that such a state $U_2$ exists, because if we set $\varrho_2:=\varrho_+ + \epsilon$ and $\epsilon>0$ sufficiently small, then all the properties are fulfilled. 
	
	Then consider the two new problems 
	\begin{align*}
	\widetilde{U}_- &= U_- \\
	\widetilde{U}_+ &= U_2,
	\end{align*}
	called problem $\sim$, and 
	\begin{align*}
	\widehat{U}_- &= U_2 \\
	\widehat{U}_+ &= U_+,
	\end{align*}
	what we call problem $ \wedge$. 
	
	Let us first consider problem $\sim$. It is easy to check that the standard solution of problem $\sim$ consists of a 1-shock and a 3-rarefaction using propostion \ref{prop:standardsolution}: We have $\varrho_-<\varrho_+$ and $\varrho_+<\varrho_2$ and hence $\varrho_-<\varrho_2$. In addition to that it holds that
	\begin{align*}
	v_{2\,2} - v_{-\,2} &= v_{+\,2} - v_{-\,2} + \sqrt{\frac{\big(\varrho_2 - \varrho_+\big)\,\big(p(\varrho_2)-p(\varrho_+)\big)}{\varrho_2\,\varrho_+}} \\
	&= -\sqrt{\frac{\big(\varrho_+ - \varrho_-\big)\,\big(p(\varrho_+)-p(\varrho_-)\big)}{\varrho_-\,\varrho_+}} + \sqrt{\frac{\big(\varrho_2 - \varrho_+\big)\,\big(p(\varrho_2)-p(\varrho_+)\big)}{\varrho_2\,\varrho_+}} \\
	&< \sqrt{\frac{\big(\varrho_2 - \varrho_+\big)\,\big(p(\varrho_2)-p(\varrho_+)\big)}{\varrho_2\,\varrho_+}} \\
	&<\int_{\varrho_-}^{\varrho_+} \frac{\sqrt{p'(r)}}{r}\,\dd r \quad<\quad \int_{\varrho_-}^{\varrho_2} \frac{\sqrt{p'(r)}}{r}\,\dd r,
	\end{align*}
	where we used \eqref{eq:temp51}, and 
	\begin{align*}
	v_{2\,2} &- v_{-\,2} = -\sqrt{\frac{\big(\varrho_+ - \varrho_-\big)\,\big(p(\varrho_+)-p(\varrho_-)\big)}{\varrho_-\,\varrho_+}} + \sqrt{\frac{\big(\varrho_2 - \varrho_+\big)\,\big(p(\varrho_2)-p(\varrho_+)\big)}{\varrho_2\,\varrho_+}} \\
	&\quad> -\sqrt{\frac{\big(\varrho_+ - \varrho_-\big)\,\big(p(\varrho_+)-p(\varrho_-)\big)}{\varrho_-\,\varrho_+}} \ >\  -\sqrt{\frac{\big(\varrho_2 - \varrho_-\big)\,\big(p(\varrho_2) - p(\varrho_-)\big)}{\varrho_2\,\varrho_-}},
	\end{align*}
	where lemma \ref{lemma:rootroot} was applied. 
	
	Hence we showed that the standard solution to problem $\sim$ consists of a 1-shock and a 3-rarefaction wave. 
	
	Because $U_2\in B_r(U_+)$ and $\varrho_2>\varrho_+$, according to lemma \ref{lemma:SRperturb} there exists an admissible fan subsolution to problem $\sim$ and hence infinitely many admissible weak solutions. Additionally the same lemma yields $\varrho_1<\varrho_+$. 
	
	Now consider problem $\wedge$. We are going to prove that the standard solution to problem $\wedge$ consists only of a 3-shock using proposition \ref{prop:standardsolution}. By definition of $U_2$ we have $\varrho_2>\varrho_+$ and additionally
	\begin{equation*}
	v_{+\,2} - v_{2\,2} = -\sqrt{\frac{\big(\varrho_2 - \varrho_+\big)\,\big(p(\varrho_2)-p(\varrho_+)\big)}{\varrho_2\,\varrho_+}}.
	\end{equation*}	
	This shows that the standard solution of problem $\wedge$ consists of a just a 3-shock. 
	
	To conclude we put together the wild solutions to problem $\sim$ and the standard solution to problem $\wedge$. To do this it remains to show that $\mu_1<\mu_2$ where $\mu_1$ is the speed of the right interface of the wild solutions of problem $\sim$ and $\mu_2=\frac{\varrho_2\,v_{2\,2} - \varrho_+\,v_{+\,2}}{\varrho_2 - \varrho_+}$ is the speed of the shock of the standard solution to problem $\wedge$. 
	
	Since we have an admissible fan subsolution, we can apply proposition \ref{prop:2alggl}. As in the proof of theorem \ref{thm:SRexistence} we get from \eqref{eq:2rhr2}
	\begin{equation} 
	\delta_1=\frac{\mu_1}{\varrho_1}\,(\varrho_1\,v_{1\,2} - \varrho_2\,v_{2\,2}) + \frac{\varrho_2}{\varrho_1}\,v_{2\,2}^2 - \frac{p(\varrho_1) - p(\varrho_2)}{\varrho_1} - v_{1\,2}^2
	\label{eq:temp52}
	\end{equation}
	and from \eqref{eq:2rhr1}
	\begin{equation*}
	v_{1\,2}=\frac{1}{\varrho_1}\big(\mu_1\,(\varrho_1-\varrho_2)+\varrho_2\,v_{2\,2}\big).
	\end{equation*}
	We use the latter to eliminate $v_{1\,2}$ in \eqref{eq:temp52} and obtain after some calculation
	\begin{equation*}
	\delta_1=\frac{\varrho_1 - \varrho_2}{\varrho_1^2}\,\varrho_2\,(\mu_1 - v_{2\,2})^2 - \frac{p(\varrho_1) - p(\varrho_2)}{\varrho_1}.
	\end{equation*}
	Since $\delta_1>0$, see \eqref{eq:2sc1}, it follows that 
	\begin{equation*}
	\frac{\varrho_1 - \varrho_2}{\varrho_1^2}\,\varrho_2\,(\mu_1 - v_{2\,2})^2 - \frac{p(\varrho_1) - p(\varrho_2)}{\varrho_1} > 0.
	\end{equation*}
	Because $\varrho_1 < \varrho_+ < \varrho_2$, we have $\varrho_1 - \varrho_2 < 0$, and hence the inequality above is equivalent to 
	\begin{equation*}
	(\mu_1 - v_{2\,2})^2 < \frac{\varrho_1}{\varrho_2}\ \ \frac{p(\varrho_1) - p(\varrho_2)}{\varrho_1 - \varrho_2}.
	\end{equation*}
	
	In addition to that we get because of $\varrho_1<\varrho_+<\varrho_2$ 
	\begin{equation*}
	\frac{\varrho_1}{\varrho_2} < \frac{\varrho_+}{\varrho_2},
	\end{equation*}
	and using the convexity of $p$
	\begin{equation*}
	\frac{p(\varrho_2) - p(\varrho_1)}{\varrho_2 - \varrho_1} \leq \frac{p(\varrho_2) - p(\varrho_+)}{\varrho_2 - \varrho_+}.
	\end{equation*}
	
	Therefore
	\begin{align*}
	\mu_1 &< v_{2\,2} + \sqrt{\frac{\varrho_1}{\varrho_2}\ \ \frac{p(\varrho_1) - p(\varrho_2)}{\varrho_1 - \varrho_2}} \\
	&\leq v_{2\,2} + \sqrt{\frac{\varrho_+}{\varrho_2}\ \ \frac{p(\varrho_2) - p(\varrho_+)}{\varrho_2 - \varrho_+}} \\
	&=\frac{v_{2\,2}\,(\varrho_2 - \varrho_+)}{\varrho_2 - \varrho_+} + \frac{\varrho_+}{\varrho_2 - \varrho_+}\sqrt{\frac{(\varrho_2 - \varrho_+)\,\big(p(\varrho_2) - p(\varrho_+)\big)}{\varrho_2\,\varrho_+}} \\
	&=\frac{v_{2\,2}\,(\varrho_2 - \varrho_+)}{\varrho_2 - \varrho_+} + \frac{\varrho_+}{\varrho_2 - \varrho_+}\,(v_{2\,2} - v_{+\,2}) \\
	&= \frac{\varrho_2\,v_{2\,2} - \varrho_+\,v_{+\,2}}{\varrho_2 - \varrho_+}\quad =\quad \mu_2,
	\end{align*}
	which is the desired inequality $\mu_1 < \mu_2$. \qed
\end{proof}

The proof of theorem \ref{thm:Sexistence} yields admissible weak solutions of the form illustrated in figure \ref{fig:S}.

\begin{figure}[hbt]
	\centering
	\includegraphics[width=0.85\textwidth]{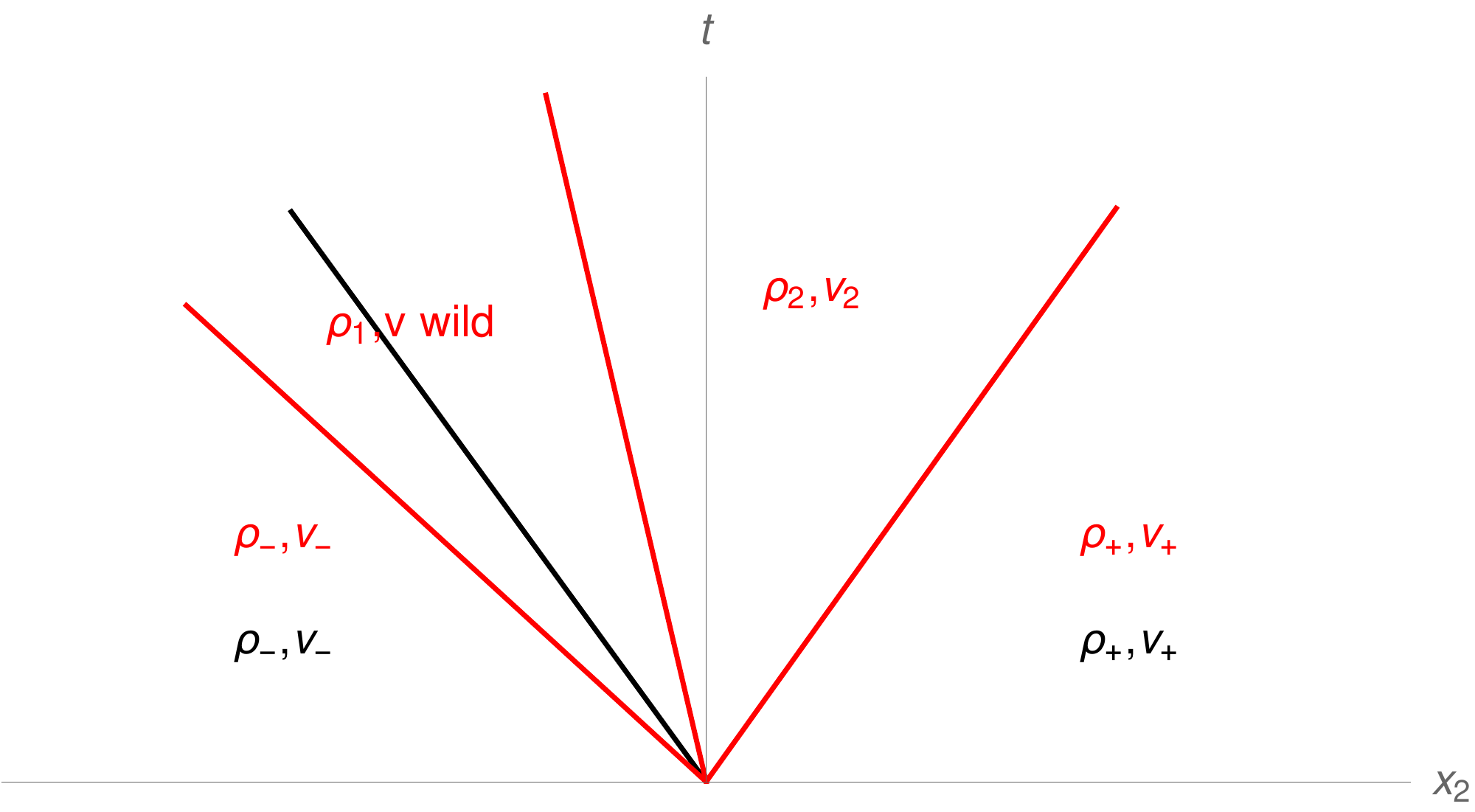}
	\caption{Structure of the standard solution (black) and of the admissible weak solutions produced in the proof of theorem \ref{thm:Sexistence} (red) } 
	\label{fig:S}
\end{figure} 

\noindent
\textbf{Acknowledgement: } The first author was partially supported by the DAAD program ``A New Passage to India'' to visit the TIFR CAM in Bangalore.

\medskip

\noindent
\textbf{Conflict of Interest:} The authors declare that they have no conflict of interest.

\bibliographystyle{spmpsci}      

\begin{thebibliography}{1}
	\providecommand{\url}[1]{{#1}}
	\providecommand{\urlprefix}{URL }
	\expandafter\ifx\csname urlstyle\endcsname\relax
	\providecommand{\doi}[1]{DOI~\discretionary{}{}{}#1}\else
	\providecommand{\doi}{DOI~\discretionary{}{}{}\begingroup
		\urlstyle{rm}\Url}\fi
	
	\bibitem{brezina}
	B\v{r}ezina, J., Chiodaroli, E., Kreml, O.: On contact discontinuities in multi-dimensional isentropic Euler equations.
	\newblock Preprint (2017), arXiv: 1707.00473
	
	\bibitem{chen}
	Chen, G.Q., Chen, J.: Stability of rarefaction waves and vacuum states for the
	multidimensional euler equations.
	\newblock J. Hyperbolic Differ. Equ. \textbf{4}(1), 105--122 (2007)
	
	\bibitem{chio15}
	Chiodaroli, E., DeLellis, C., Kreml, O.: Global ill-posedness of the isentropic
	system of gas dynamics.
	\newblock Comm. Pure Appl. Math. \textbf{68}(7), 1157--1190 (2015)
	
	\bibitem{chio14}
	Chiodaroli, E., Kreml, O.: On the energy dissipation rate of solutions to the
	compressible isentropic euler system.
	\newblock Arch. Ration. Mech. Anal. \textbf{214}(3), 1019--1049 (2014)

	\bibitem{chio17}
	Chiodaroli, E., Kreml, O.: Non-uniqueness of admissible weak solutions to the Riemann problem for the isentropic Euler equations.
	\newblock Preprint (2017), arXiv: 1704.01747
	
	\bibitem{dafermos}
	Dafermos, C.M.: Hyperbolic Conservation Laws in Continuum Physics, 4 edn.
	\newblock Springer (2016)
	
	\bibitem{dls09}
	DeLellis, C., Sz{\'{e}}kelyhidi Jr., L.: The euler equations as a differential
	inclusion.
	\newblock Ann. of Math. (2) \textbf{170}(3), 1417--1436 (2009)
	
	\bibitem{dls10}
	DeLellis, C., Sz{\'{e}}kelyhidi Jr., L.: On admissibility criteria for weak
	solutions of the euler equations.
	\newblock Arch. Ration. Mech. Anal. \textbf{195}(1), 225--260 (2010)
	
	\bibitem{feir15}
	Feireisl, E., Kreml, O.: Uniqueness of rarefaction waves in multidimensional
	compressible euler system.
	\newblock J. Hyperbolic Differ. Equ. \textbf{12}(3), 489--499 (2015)
	
	\bibitem{leveque}
	LeVeque, R.J.: Finite-Volume Methods for Hyperbolic Problems.
	\newblock Cambridge (2004)
	
\end{thebibliography}


\end{document}